\documentclass[12pt,reqno,letterpaper]{amsart}
\usepackage{graphicx} 
\usepackage{thmtools}
\usepackage{amsmath,amssymb,amsthm}
\usepackage[left=2cm,top=2cm,right=2cm,bottom=2cm]{geometry}
\usepackage[normalem]{ulem}
\usepackage{mathtools}
\usepackage{enumerate}
\usepackage{comment}
\usepackage[table]{xcolor} 

\usepackage{xspace}
\usepackage{derivative}
\usepackage{bbm}
\usepackage{graphicx}
\usepackage{tikz}
\usetikzlibrary{arrows.meta}
\usetikzlibrary{math}
\usetikzlibrary{external}
\usepackage[a-1b]{pdfx}  
\usepackage[nameinlink]{cleveref}

\newtheorem{theorem}{Theorem}[section]
\newtheorem{lemma}[theorem]{Lemma}
\newtheorem{corollary}[theorem]{Corollary}
\newtheorem{remark}[theorem]{Remark}
\newtheorem{claim}[theorem]{Claim}
\newtheorem{subclaim}[theorem]{Subclaim}
\newtheorem{definition}[theorem]{Definition}
\newtheorem{proposition}[theorem]{Proposition}
\newtheorem{conjecture}[theorem]{Conjecture}



\newcommand{\eps}{\varepsilon}

\newcommand\B{{\mathcal B}}
\newcommand{\F}{{\mathcal F}}
\newcommand{\G}{{\mathcal G}}
\newcommand{\GG}{\widetilde{\mathcal G}}
\newcommand{\seq}{\subseteq}

\newcommand{\tX}{\widetilde{X}}
\newcommand{\hbeta}{\widehat{\beta}}

\newcommand{\hX}{\widehat{X}}

\newcommand{\simeps}[1]{\sim_{#1}}

\newcommand{\geBy}[1]{\mathnormal{\stackrel{(\ref{#1})}{\ge}}}
\newcommand{\leBy}[1]{\mathnormal{\stackrel{(\ref{#1})}{\le}}}

\newcommand{\geByM}[1]{\mathnormal{\stackrel{#1}{\ge}}}

\newcommand{\leByM}[1]{\mathnormal{\stackrel{#1}{\le}}}

\title{Probabilistic Counting Lemma for $K_4$}

\author[Warach Veeranonchai]{Warach Veeranonchai\\
Royal Holloway, University of London}
\address[Warach Veeranonchai]{Department of Mathematics,
Royal Holloway, University of London Egham Hill, Egham, Surrey
TW20 0EX, United Kingdom}
\email{pmah079@live.rhul.ac.uk}

\begin{document}
\begin{abstract}
The sparse analogue of Szemerédi’s regularity method has played a central role in the development of extremal results for random graphs. While the sparse embedding lemma (the K{\L}R conjecture) has been resolved, the sparse version of the counting lemma remains widely open. The conjecture, formulated by Gerke, Marciniszyn, and Steger, states that for every fixed graph $H$ and any $\beta>0$, there exists $\eps>0$ such that the following holds. Consider a balanced blow-up of $H$ with vertex classes of size $n$, where each pair corresponding to an edge of $H$ forms an $(\eps)$-regular bipartite graph with exactly $m$ edges. Assume that $m$ is above the natural threshold $m \gg n^{2-1/m_2(H)}$, then all but a $\beta^m$ proportion of such graphs contain at least $(1-\delta)$ times the expected number of copies of $H$. 
In this paper, we establish the $H=K_4$ case of the conjecture. 
\end{abstract}

\maketitle

\section{Introduction}

The celebrated Regularity Lemma of Szemerédi is a cornerstone of modern extremal graph theory. Roughly speaking, it asserts that the vertex set of every sufficiently large graph can be partitioned into a bounded number of classes such that the edges between most pairs of classes are distributed in a pseudorandom fashion. 
In many applications, the existence of a fixed subgraph is required; statements guaranteeing such existence are known as embedding lemmas. One may also be interested in determining the number of such subgraphs; statements providing asymptotically correct counts are known as counting lemmas. However, the traditional method only works when the graph is dense.

Extending this methodology to sparse graphs has been a driving theme in probabilistic combinatorics over the past two decades. The sparse version of Szemerédi's regularity lemma was independently introduced by Kohayakawa \cite{SparseRegularity} and Rödl, using $(\eps)$-regularity: a variant of $\eps$-regularity suitable for sparse graphs. However, the embedding and counting lemmas do not automatically extend to the sparse regime. 

The first problem arises when the edge density is too small. It is well known that if $p \ll n^{-1/m_2(H)}$ (\Cref{balance}), then with high probability $G(n,p)$ contains no copy of $H$ (\cite[Theorem 8]{threshold}, or see for example \cite[Theorem 5.3]{thresholdNew} for a proof in modern notation).
This naturally leads to a condition that edge density must be at least $d \gg n^{-1/m_2(H)}$.
The second problem is that, even when the density exceeds this threshold, there exist regular graphs with density above this threshold that contain no copy of $H$ \cite{CounterExample}, \cite[Example 3.10]{Application}. Consequently, it is necessary to allow that some regular graphs may fail to contain $H$. 

The sparse counterpart of the embedding lemma formulated by Kohayakawa, Łuczak, and Rödl \cite{KLR} (\Cref{thm:KLR}), known as the K{\L}R conjecture, makes this idea precise by quantifying the proportion of regular graphs that may fail to contain a copy of $H$. The conjecture was formulated as a key step towards extending regularity-based methods from dense graphs to the sparse and random setting, and it has played an important role in the development of extremal results for random graphs. In particular, it enables the use of the regularity method for extremal problems in random graphs in a manner analogous to the dense case and yields many sparse extremal consequences, see \cite{Application}. After several partial results (\cite{Master, SmallSubsets, ProbCountLemma, K4, K5, Triangle}), the conjecture was eventually proved by Balogh, Morris, and Samotij \cite{KLRproof}, and independently by Saxton and Thomason \cite{Container}, through the development of the hypergraph container method.

A sparse counterpart of the counting lemma was proposed by Gerke, Marciniszyn, and Steger \cite{ProbCountLemma} (\Cref{conj: count}). It is usually referred to as a probabilistic counting lemma, to differentiate it from other sparse versions that require stronger pseudorandom hypotheses.
Instead of guaranteeing the existence of a single copy, they conjectured that one can find almost the expected number of copies, as in a random graph with the same edge density.
The conjecture has three main components: it holds for all edge densities above $d \gg n^{-1/m_2(H)}$, the failure probability is at most $\beta^m$ for arbitrarily small $\beta$ (for some small $\eps$ depending on $\beta$), and the number of copies of $H$ is at least $(1-\delta)$ of the expected number. Each of these three components is essentially the best possible.

The conjecture remains open, but several partial results have been obtained. A version with the correct density condition, but only guaranteeing a fixed positive proportion of the expected number and with a failure probability larger than conjectured, was proved in \cite{ONklr}. If one further assumes that $H$ is strictly balanced (\Cref{balance}), the same paper also proves that the correct number of copies up to a $(1-\delta)$ factor can be obtained, again with a failure probability that is too large.
When Saxton and Thomason \cite[Theorem 10.2]{Container} proved the K{\L}R conjecture using the container method, they in fact established a version of the probabilistic counting lemma with the correct density condition and failure probability, but still only guaranteed a fixed positive proportion of the expected number. 
Nenadov \cite{KLRnew} later gave a direct proof of the same statement using induction.

Gerke, Marciniszyn, and Steger \cite{ProbCountLemma} also proved a version with the correct number of copies and the optimal failure probability for complete graphs $K_\ell$, $\ell \ge 3$, but under the stronger density assumption $d \gg n^{-1/(\ell-1)}$ instead of $d \gg n^{-2/(\ell+1)}$. Consequently, for $H = K_3$, the full probabilistic counting conjecture is known to hold. This is the only value of $\ell$ for which the case $H = K_\ell$ has been fully established.

Their key observation in \cite{ProbCountLemma} is that, for almost all $(\eps)$-regular graphs between $\ell$ vertex classes, the induced graphs obtained by restricting to the neighbourhood of a typical vertex behave similarly to $(\eps)$-regular graphs between $\ell-1$ classes of the corresponding sizes (we later generalise this lemma of \cite{ProbCountLemma} to \Cref{lem:gen_small_bad}).
This allows one to restrict attention to the neighbourhood of a single vertex and apply induction, provided that the edge density relative to the reduced vertex classes remains above the value required for the smaller case. However, when the original density is close to the threshold, this condition fails. This is why the result in \cite{ProbCountLemma} applies only to densities in the regime around the square root of the threshold and above. 

After submitting this paper, we became aware that Nenadov had already established the $K_4$, $K_5$, and $K_6$ cases of the probabilistic counting conjecture in his unpublished master's thesis from 2012. Nenadov's argument is long and technical, even in the $K_4$ case. By contrast, although our proof requires a technical generalisation of the main lemma from \cite{ProbCountLemma}, the main idea is more concise and intuitive.
Furthermore, our approach allows the vertex sets to have different sizes. We believe that this additional flexibility may be useful in future extensions to larger complete graphs.
In this paper, we give a shorter proof of the $H=K_4$ case of the probabilistic counting conjecture, \Cref{conj: count}. Our idea is as follows.

Let $X_1$, $X_2$, and $X_3$ be three vertex sets, and suppose $G$ consists of a bipartite graph $G_1$ between $X_1$ and $X_2$ of density $d_1$, together with a bipartite graph $G_2$ between $X_1$ and $X_3$ of density $d_2$. From $G=G_1 \cup G_2$, we construct an auxiliary bipartite graph $G'$ between $X_1$ and the set $X_2 \times X_3$, where $(x_1,(x_2,x_3))$ is an edge of $G'$ if and only if $x_1x_2 \in E(G_1)$ and $x_1x_3 \in E(G_2)$. We show that for any $(\eps,d_1)$-lower-regular $G_1$ (\Cref{def:eps_p_regular}), almost all $(\eps,d_2)$-lower-regular graphs $G_2$ have the property that, provided $|X_1| \gg (d_1 d_2)^{-1}$, the auxiliary graph $G'$ is $(\eps',d_1d_2)$-lower-regular (\Cref{lem:Auxiliary_regular}). 

For $m  \gg (d_1 d_2)^{-1}$, \Cref{thm:low reg_her} tell us that almost all collections of $m$ edges between $X_2$ and $X_3$ are $(\eps'',d_1d_2)$-lower-regular with $X_1$. Note that in practice, we first fix any $\eps''>0$, find $\eps'$ small enough so that the lemma yields our $\eps''$, then choose $\eps$ small enough so that the lemma yields our $\eps'$, and so on.
We can use this to convert the property that most of the $m$ edges between $X_2$ and $X_3$ belong to at least (almost) the expected amount of $K_3$ in $X_1 \cup X_2 \cup X_3$ into the property that most vertices in $X_1$ belong to the expected amount of $k_3$. 

The advantage of this approach is that only the size of the vertex set $X_1$ needs to exceed $d^{-2}$ (when $d_1=d_2=d$), while the sizes of $X_2$ and $X_3$ need only exceed $d^{-1}$ (we also need $m \gg d^{-2}$). For the case $\ell=4$, we consider the vertex sets $V_1, V_2, V_3, V_4$. Rather than ``seeing'' all the edges between all vertex sets, we first leave out the edges between $V_1$ and $V_4$. 
For a typical vertex $v_1 \in V_1$, its neighbourhood in $V_2$, $V_3$, and the entirety of $V_4$ contains the expected number of copies of $K_3$. More importantly, almost all vertices in $V_4$ are also contained in the expected number of copies of $K_3$ (the exception can be different for different vertices in $V_1$). This implies that for most pairs $(v_1,v_4)\in V_1 \times V_4$, the edge $v_1v_4$, if present, will belong to nearly the expected number of $K_4$. Finally, we choose the edge between $V_1$ and $V_4$, and almost all ways to choose $m$ edges between them will give almost the expected number of $K_4$ overall.

There is still a subtle issue to address. The corresponding lemma from \cite{ProbCountLemma} cannot be applied directly to $V_4$ and the neighbourhood of $v_1$ in $V_2 \cup V_3$, since the relevant vertex sets are no longer all of the same size. This is precisely why the generalisation stated as \Cref{lem:gen_small_bad} is needed, together with corresponding extensions of the relevant definitions. Although the proof follows the argument of \cite{ProbCountLemma} closely, it is lengthy and involves several technical modifications to accommodate unequal vertex sets. We therefore present it in a separate section, allowing readers to skip it if they wish.

The rest of this paper is organised as follows. In \Cref{section:def}, we compile existing definitions and known results. In \Cref{section:main}, we prove the $H=K_4$ case of \Cref{conj: count}, assuming \Cref{lem:gen_small_bad} holds. The proof of \Cref{lem:gen_small_bad} is presented in \Cref{section:lemma_proof}. 

\subsection{Definitions}\label{section:def}

\begin{definition} \label{def:eps_p_regular}
For a graph $G=(V,E)$ and two disjoint sets $V_1, V_2\seq V$, we denote the set of edges with one endpoint in $V_1$ and one endpoint in $V_2$ by $E(V_1, V_2)$. The \emph{density} $d(V_1, V_2)$ is defined as 
$d(V_1,V_2) = {|E(V_1, V_2)|}/(|V_1||V_2|).$
We say that the graph induced by $V_1,V_2$ is \emph{$(\eps)$-regular} if for all $V_1' \subseteq V_1$ and $V_2' \subseteq V_2$  with  $|V_1'| \ge \eps |V_1|$ and $|V_2'| \ge \eps |V_2|$, we have
\[
|d(V_1', V_2') - d(V_1, V_2)| \le \eps d(V_1,V_2) .
\]
We say that the graph induced by $V_1,V_2$ is \emph{$(\eps,d)$-lower-regular} if for all $V_1' \subseteq V_1$ and $V_2' \subseteq V_2$  with  $|V_1'| \ge \eps |V_1|$ and $|V_2'| \ge \eps |V_2|$, we have
\[
d(V_1', V_2')\ge (1-\eps)d.
\]
\end{definition}

\begin{definition}\label{def:regular_tuples}
For a graph $H$, let $\G(H, n, m)$ be the family of graphs on the vertex set $V=\bigcup_{x\in V(H)} V_x$, where the sets $V_x$ are pairwise disjoint sets of vertices of size $n$, and edge set $E=\bigcup_{\{x,y\}\in E(H)}E_{xy}$, where $E_{xy} \seq V_x \times V_y$ and $|E_{xy}| = m$. Let $\G(H, n, m, \eps) \seq \G(H, n, m)$ denote the set of graphs in $\G(H, n, m)$ satisfying that each $(V_x\cup V_y,E_{xy})$ is an $(\eps)$-regular graph.
\end{definition}

\begin{definition}\label{balance}
The \emph{2-density} of a graph $H$ is defined as
\[
m_2(H) = \max\left\{\frac{|E(H')|-1}{|V(H')|-2}: H'\subseteq H, |V(H')|\geq 3\right\}.
\] 
A graph $H$ is called \emph{balanced} if 
\[
m_2(H)=\frac{|E(H)|-1}{|V(H)|-2}.
\] 
It is called \emph{strictly balanced} if $m_2(H)>m_2(H')$ for every proper subgraph $H' \subsetneq H$.
\end{definition}

\begin{theorem}\cite{KLRproof,Container}[formally known as K{\L}R-Conjecture \cite{KLR}]\label{thm:KLR}
Let $H$ be a fixed graph and let
\[ 
  \F(H,n,m)= \{G \in \G(H, n, m): H \text{ is not a subgraph of } G\}.
\]
For any $\beta > 0$, there exist constants $\eps > 0$, $C>0$, $n_0>0$ such that for all $m \geq Cn^{2-1/m_2(H)}$ and~$n\geq n_0$, we have
\[
|\F(H,n,m) \cap \G(H, n, m, \eps)| \leq \beta^m {n^2 \choose m}^{|E(H)|}.
\]
\end{theorem}

\begin{definition}\label{def: degree}
Fix a graph $H$. For a graph $G$ with vertex set $V=\bigcup_{x\in V(H)} V_x$, we say that a copy of $H$ in a graph $G$ is \emph{canonical} if each vertex $x \in V(H)$ belongs to $V_x$. Given an edge $e=(v_a,v_b)\in V_a \times V_b$ for some distinct $a,b \in V(H)$, we denote by $\deg_H(e,G)$ the number of canonical copies of $H$ in $G$ which contain $v_a$ and $v_b$. Similarly, for a vertex $v$, we denote by $\deg_H(v, G)$ the number of canonical copies of $H$ in $G$ which contain $v$.
\end{definition}

\begin{definition}\label{def:forbidden tuples}
Let~$\F(H, n, m, \delta) \seq \G(H, n, m)$ denote the family of graphs that contain fewer than
\[
(1 - \delta)n^{|V(H)|}\left(\frac{m}{n^2}\right)^{|E(H)|}
\]
canonical copies of $H$.   
\end{definition}

\begin{conjecture}[Counting Lemma \cite{ProbCountLemma}]\label{conj: count} 
Let $H$ be a fixed graph. For any $\beta > 0$ and $\delta>0$, there exist constants $\eps > 0$, $C>0$, and~$n_0 >0$ such that for all $m \geq Cn^{2-1/m_2(H)}$ and~$n\geq n_0$, we have
\[ 
|\F(H,n,m,\delta) \cap \G(H, n, m, \eps)| \leq \beta^m {n^2 \choose m}^{|E(H)|}.
\]
\end{conjecture}

Let $\Gamma(v)$ denote the neighbourhood of vertex $v$. When we have partition classes $V_i$ for various $i$, we write $\Gamma_i(v):=\Gamma(v) \cap V_i$.
\begin{proposition}\label{prop:reg_degree}
Let $G=(V_1\cup V_2, E)$ be
an $(\eps)$-regular graph with density $d$, and let
$V_2'\subseteq V_2$ satisfy 
$|V_2'|\geq \eps |V_2|$. Then at most $\lfloor \eps|V_1| \rfloor$ vertices 
$v\in V_1$ do not satisfy
\begin{equation*}
|\Gamma(v)\cap V'_2|\geq  (1-\eps)d |V_2'|, 
\end{equation*}
and at most $\lfloor \eps |V_1|\rfloor$ vertices $u\in V_1$ do not satisfy
\[ 
|\Gamma(u)\cap V'_2|\leq (1+\eps) d  |V_2'|. 
\]
Moreover, let $G'=(V_3\cup V_4, E')$ be an $(\eps,d)$-lower-regular graph, and let $V_4'\subseteq V_4$ satisfy $|V_4'|\geq \eps |V_4|$. Then at most $\lfloor \eps|V_3| \rfloor$ vertices $v\in V_3$ do not satisfy
\begin{equation*}
|\Gamma(v)\cap V'_4|\geq  (1-\eps)d |V_4'|.
\end{equation*}
\end{proposition}

It follows directly from the definition and the triangle inequality that large subsets of vertices inherit the regularity.
We write~$a \simeps{\eps} b$ if~$(1 - \eps) b \leq a \leq (1 + \eps)b$.

\begin{proposition}\label{prop: reg_large_her}
Let $G=(V_1\cup V_2,E)$ be an $(\eps)$-regular graph of density $d$. Then any subset $V_1'$ of $V_1$ of size at least $\alpha|V_1|$ with $\alpha>\eps>0$ induces an $(\eps')$-regular graph with $\eps'=\max\{\eps/\alpha, 2\eps/(1-\eps) \}$ and density $d' \simeps{\eps} d$.

Moreover, let $G'=(V_3\cup V_4, E')$ be an $(\eps,d)$-lower-regular graph. Then any subset $V_3'$ of $V_3$ of size at least $\alpha|V_3|$ with $\alpha>\eps>0$  induces an $(\eps',d')$-lower-regular graph with $\eps'=\max\{\eps/\alpha, 2\eps/(1-\eps) \}$ and $d' \simeps{\eps} d$.
\end{proposition}

By a probabilistic argument one can show that an~$(\eps)$-regular pair with~$|E(B)|$ edges contains a $(2\eps)$-regular subgraph with exactly~$m \leq |E(B)|$ edges provided~$m \gg |V(G)|$.
\begin{lemma}[\cite{Application} Lemma 4.3]\label{lem:sub_of_eps_is_eps}
For all $0<\eps\leq 1/6$, there exists a constant $C=C(\eps)$ such that any $(\eps)$-regular graph $B=(V_1\cup V_2,E)$ contains a $(2\eps)$-regular subgraph with $m$ edges for all $m$ satisfying $C|V(B)|\leq m \leq |E(B)|$.
\end{lemma}

\begin{theorem}[\cite{SmallSubsets} Theorem 3.6]\label{thm:low reg_her}
For $0<\beta,\eps'<1$ there exist $\eps=\eps(\beta,\eps')>0$ and
$C=C(\eps')$ such that every $(\eps,d)$-lower-regular graph $G=(V_1\cup V_2,E)$ satisfies that the number of sets $Q \subset V_1$ of size $q=|Q|\geq C d^{-1}$ that form an $(\eps',d)$-lower-regular with $V_2$, is at least
\[
\left(1-\beta^q\right)\binom{|V_1|}{q}.
\]
\end{theorem}

\begin{theorem}[\cite{SmallSubsets} Theorem 3.7]\label{thm:reg_her}
For $0<\beta,\eps'<1$, there exist $ \eps = \eps(\beta,\eps')>0$ and $C=C(\eps')$ such that every $(\eps)$-regular graph $G=(V_1\cup V_2,E)$ with density $d$ satisfies that the number of sets $Q \subset V_1$ of size $q=|Q|\geq C d^{-1}$ that contain a set $\widetilde{Q}$ of size at least $(1 - \eps')|Q|$ forming an $(\eps')$-regular graph of density $d' \simeps{\eps} d$ with $V_2$, is at least 
\[ 
(1-\beta^q)\binom{|V_1|}{q}.
\]
\end{theorem}

Finally, we collect several auxiliary results that will be used throughout the paper. 
If $0 \leq x \leq 1$, then
\begin{align}\label{eq:binomial coefficient 1}
  {xa \choose b} &\leq {a \choose b} x^b . \\
\intertext{If~$b \leq a$ and both are integral, then}
	\label{eq:binomial coefficient 2}
  {a \choose b - c}{a \choose c} &\leq 4^b {a \choose b} .\\
\intertext{For all integral $a$, $b$, $c$, and~$d$, we have}
	\label{eq:binomial coefficient 3}
	{a \choose b}{c \choose d}&\leq {a+c \choose b+d} .
\end{align}

\section{Main}\label{section:main}
In our proof, we require a version of the lemma from \cite{ProbCountLemma} in which the vertex sets $V_x$ are allowed to have different sizes, which in turn necessitates corresponding generalisations of earlier definitions. \Cref{def:gen_family} through \Cref{lem:gen_small_bad} provide these extensions. \Cref{def:Aux graph} to \Cref{lem:small family} will show that the family of graphs that don't have our desired property is small, with \Cref{lem:Auxiliary_regular} being our key result. Finally, combining \Cref{lem:gen_small_bad} and \Cref{lem:small family}, we prove the $H=K_4$ case of \Cref{conj: count}.

First, we define a blow-up of a graph where the vertex sets do not necessarily have the same size. 
\begin{definition}\label{def:gen_family}
For a graph $H$ with $\ell=|V(H)|$, and an $\ell \times\ell$ symmetric matrix $D$ with entries in $(0,1]$, let $\G(H, n_1,\dots,n_\ell, D)$ be the family of graphs on the vertex set $V=\bigcup_{x=1}^{\ell} V_x$, where the sets $V_x$ are pairwise disjoint sets of vertices of size $|V_x|=n_x$ for $x=1,\dots,\ell$, and edge set $E=\bigcup_{\{x,y\}\in E(H)}E_{xy}$, where $E_{xy} \seq V_x \times V_y$ and $|E_{xy}| = D_{x,y} n_x n_y$. 
Let $\G(H, n_1,\dots,n_\ell, D, \eps) \seq \G(H, n_1,\dots,n_\ell, D)$ denote the set of graphs in $\G(H, n_1,\dots,n_\ell, D)$ satisfying that each $(V_x\cup V_y,E_{xy})$ is an $(\eps)$-regular graph. 
\end{definition}

Usually, $D$ will be a constant matrix with all entries being $d$. In that case, we may write $\G(H, n_1,\dots,n_\ell, d)$ and $\G(H, n_1,\dots,n_\ell, d, \eps)$ instead. 
For convenience, we write $\G(\ell, n_1,\dots,n_\ell, D)$ instead of $\G(K_\ell, n_1,\dots,n_\ell, D)$ and $\G(\ell, n_1,\dots,n_\ell, D, \eps)$ instead of $\G(K_\ell, n_1,\dots,n_\ell, D, \eps)$. 
For an edge $e \notin E(H)$ with endpoints $a,b$ and $G\in\G(H,n_1,\dots,n_\ell,D)$, we say that a graph $G'\in\G(H+e,n_1,\dots,n_\ell,D)$ is an \emph{extension} of $G$ if $G$ can be obtained from $G'$ by deleting all edges between $V_a$ and $V_b$.

Later, we will construct graphs that do not contain some desired properties. To do so, we need to be slightly careful about the order in which we extend the graphs. The next definition makes this more precise.
\begin{definition}
Let $H = K_\ell$ with vertex set $v_1,\dots,v_\ell$.
For $1\le k\le \ell-1$, let $E_k(H)$ denote the set of edges between
$v_{k+1}$ and $v_1,\dots,v_k$.

A sequence of graphs
\[
H_0 \subset H_1 \subset \dots \subset H_{\binom{\ell}{2}}
\]
is called \emph{valid for $H$} if the following hold.

\begin{enumerate}[(i)]
\item $H_0=\emptyset$ and $H_{\binom{\ell}{2}}=H$.

\item For each $1\le i\le \binom{\ell}{2}$ we have
\[
H_i = H_{i-1} + e_i
\]
for some $e_i\in E(H)$.

\item The edges are pairwise distinct.

\item For every $1\le k<\ell$ and
\[
\binom{k}{2} < i \le \binom{k+1}{2},
\]
the edge $e_i$ lies in $E_k(H)$.
\end{enumerate}

For $1\le k\le \ell-1$ and $1\le i\le k$, write
\[
H_{k,i} := H_{\binom{k}{2}+i},
\qquad
e_{k,i} := e_{\binom{k}{2}+i}.
\]
Then each $e_{k,i}$ has the form
\[
e_{k,i}=\{v_{\sigma_k(i)},v_{k+1}\},
\]
where $\sigma_k$ is a permutation of $\{1,\dots,k\}$.
\end{definition}

In other words, a valid sequence builds $K_\ell$ from the empty graph by
adding edges one at a time so that all edges incident to $v_{k+1}$ are added
only after all edges among $\{v_1,\dots,v_k\}$ have already appeared.

Next, we define a generalisation of what it means for a family of graphs to be \emph{small}.
Given two $k\times k$ matrices $D$ and $D'$, we write $D \ge D'$ if
$D_{i,j} \ge D'_{i,j}$ for all $1 \le i,j \le k$.

\begin{definition}\label{def:gen_small} 
Fix $\ell$ and let $H=K_\ell$.  
Let $\mathbb H: H_0 \subset H_1 \subset \dots \subset H_{\binom{\ell}{2}}$
be a valid sequence for $H$.

A family
\[
\B(H,n_1,\dots,n_\ell,D)
\subseteq
\G(H,n_1,\dots,n_\ell,D)
\]
is said to be \emph{small} with respect to a density matrix $D_0$, functions $n_1(n),n_2(n),\dots,n_\ell(n)$ and the sequence $\mathbb H$ if the following holds.

For every $\beta>0$ there exist constants
\[
n_\beta\in\mathbb N,\qquad C_\beta>0,\qquad \eps_\beta>0
\]
and families
\[
\B(H_{k,i},n_1,\dots,n_\ell,D)
\subseteq
\G(H_{k,i},n_1,\dots,n_\ell,D)
\qquad (1\le k\le \ell-1,\;1\le i\le k)
\]
such that the following properties hold.

Let
\[
\B(H_{k,i},n_1,\dots,n_\ell,D,\eps):= \B(H_{k,i},n_1,\dots,n_\ell,D)
\cap \G(H_{k,i},n_1,\dots,n_\ell,D,\eps),
\]
and write $m_{a,b}=D_{a,b}n_an_b$. For all
\[
0<\eps\le\eps_\beta,\qquad
n\ge n_\beta,\qquad
D\ge C_\beta D_0,
\]
the following conditions hold.
\begin{enumerate}[(i)]
\item\label{def:gen_small i}
\[
|\B(H_{1,1},n_1,\dots,n_\ell,D,\eps)| \le \beta^{m_{1,2}} \binom{n_1n_2}{m_{1,2}}.
\]

\item\label{def:gen_small ii} For $1 \le k \le \ell-1$ and $1 \le i \le k$ such that $(i,k) \neq (1,1)$, 
\[
\begin{split}
G \in \G(H_{k,i-1},n_1,\dots,n_\ell,D,\eps) \setminus \B(H_{k,i-1},n_1,\dots,n_\ell,D,\eps) \; \Rightarrow\; \\
\left|\left\{\, G' \in \B(H_{k,i},n_1,\dots,n_\ell,D,\eps) :\, G' \text{ extends } G \,\right\}\right|
\le \beta^{m_{\sigma_k(i),k+1}} \binom{n_{\sigma_k(i)} n_{k+1}}{m_{\sigma_k(i),k+1}} .
\end{split}
\]
\end{enumerate}

We say that $\B(H,n_1,\dots,n_\ell,D)$ is \emph{small with respect to $D_0$}
if it is small with respect to $D_0$, $n_1,n_2,\dots,n_\ell$ and some valid sequence $\mathbb H$ of $H$. If $D_0$ is the constant matrix with density $d_0$, we say that the family is
\emph{small with respect to $d_0$}.
\end{definition}

\begin{remark}
Usually, $D_0$ is a constant matrix with entries $d_0=n^{-1/m_2(H)}$. Moreover, as a result of \Cref{def:gen_small}, 
we have
\[
|\B(H,n_1,\dots,n_\ell,D,\eps)| \le (2\beta)^{\min m_{a,b}}|\G(H,n_1,\dots,n_\ell,D,\eps)|.
\]
If a family of graphs is small in the sense of the original definition in \cite{ProbCountLemma}, then [\cite{ProbCountLemma} Claim 16] shows that it is also small in our definition with $D_0$ a constant matrix with entries $d_0=m_0(n)/n^2$.
\end{remark}

Now we can precisely state our generalised version.
\begin{lemma}\label{lem:gen_small_bad}
Let $K_\ell-e$ be the graph obtained by deleting the edge $(1,2)$ from the complete graph on vertices $1,2,\dots,\ell$. Let $x_2=n$, and $x_i=x$ for $3\le i \le \ell$, let $d_0(x_2,\dots,x_\ell)$ be a monotone increasing function and let $\B(\ell - 1, x_2,\dots,x_{\ell},d')$ be small with respect to $d_0$. 
For all $\beta > 0$ and $\eps' > 0$, there exist constants $\eps > 0$ and $C$ such that for $m,n,x$ and $d'$ satisfying,
\[
x = (1 - \eps')\frac{m}{n}\enspace, \quad 
\frac{m}{n^2} \ge C d_0(x) \ge \frac{1}{x}\enspace, \quad
m \geq 2 n^{3/2} \sqrt{\log n}\enspace, \quad
(1 - \eps')\frac{m}{n^2} \ge d' \ge  \eps'\frac{m}{n^2}
\enspace,
\] 
and~$n$ sufficiently large, all but at most $\beta ^m{n^2\choose m}^{{\ell\choose 2}-1}$ graphs $G \in \G(K_\ell-e,n,m,\eps)$ satisfy the following property: there exist at least $(1 - \eps')n$ vertices $v_1 \in V_1$ such that $V_2 \cup \left(\cup_{i=3}^\ell \Gamma_1(v_i) \right)$ contains a member of the family
\[
\G\left(\ell-1, x_2,\dots,x_\ell,d' ,\eps'\right) \setminus \B(\ell-1, x_2,\dots,x_\ell,d') 
\]
as a subgraph.
\end{lemma}
We deferred the proof of \Cref{lem:gen_small_bad}, which is a generalisation of the original version, to \Cref{section:lemma_proof}.

\begin{definition}\label{def:Aux graph}
Let $G=X_1\cup X_2\cup X_3$ be a graph with pairwise disjoint vertex sets $X_1, X_2, X_3$. We create an auxiliary graph $A(G)$ as a bipartite graph between the vertex sets $X_1$ and $Y=X_2\times X_3$ with the edge set $E(A(G))=\{(x_1,y) \; : \; y=(x_2,x_3)\in Y, x_i\in X_i, (x_1,x_2),(x_2,x_3) \in E(G)\}.$ We called $A(G)$ the \emph{path-Aux graph between $X_1$ and $X_2 \times X_3$}.
\end{definition}

Our key observation is that, under suitable conditions, most path-Aux graphs are lower-regular.
\begin{lemma}\label{lem:Auxiliary_regular}
For $0 < \beta,\eps' < 1$, there exist constants $\eps,C>0$ such that the following holds. Let $X_1$, $X_2$, and $X_3$ be three vertex sets of size $n_1,n_2,n_3$, respectively. Let $d_1,d_2$ satisfy
\[
n_1\ge \frac{C}{d_1d_2} \quad, \enspace n_2 \ge \frac{C}{d_1} \quad, \enspace n_3 \ge \frac{C}{d_2}.
\]
Then for any $(\eps,d_1)$-lower-regular graph $G_1$ between $X_1$ and $X_2$, at most $\beta^{d_2n_1n_3} \binom{n_1n_2}{d_2n_1n_3}$ $(\eps,d_1)$-lower-regular graphs $G_2$ between $X_1$ and $X_3$ have a property that the path-Aux graph between $X_1$ and $X_2 \times X_3$ is not $(\eps',d_1d_2)$-lower-regular.
\end{lemma}
\begin{proof}
Let $\eps_1, \beta_0$ be small enough so that 
\[
(1-3\eps_1/\eps')(1-3\eps_1) \ge (1-\eps') \quad, \enspace 4\beta_0 ^{(1-\eps_1)\eps_1^2} <\beta.
\]
Apply \Cref{thm:low reg_her} with $\eps' \leftarrow \eps_1, \beta \leftarrow \beta_0$ yields constants $\eps \leftarrow \eps_2, C_0 \leftarrow C$.  
We can assume $\eps_2 \le \eps_1 < 1/4$. Set
\[
C := \frac{C_0}{(1-\eps_1)^3\eps_1}\quad, \enspace \eps := \eps_2 \eps_1 \le \frac{\eps_2}{4}
\]
This implies that
\[
\eps_2 \ge \max\{\eps/\eps_1, 2\eps/(1-\eps) \} \quad, \enspace \eps \le \eps_1.
\]

Fix any $(\eps,d_1)$-lower-regular graph $G_1$ between $X_1$ and $X_2$. Write $m_2:= d_2n_1n_3$.
We first prove the following claim.
\begin{claim}
For any $X_1'\subset X_1$ of size at least $\eps_1 X_1$, in all but $(\beta/2)^{m_2}\binom{n_1 n_3}{m_2}$ $(\eps,d_2)$-lower-regular graphs $G_2$ with $m_2$ edges between $X_1$ and $X_3$, at least $(1-3\eps_1)|Y|$ vertices in $Y=X_2\times X_3$ have degree at least 
$(1-3\eps_1) d_1d_2|X_1'|$ to $X_1'$ in $A(G_1 \cup G_2)$.  
\end{claim}

\begin{proof}
    By \Cref{prop: reg_large_her}, $G_1[X_1' \cup X_2]$ is an $(\eps_2,d_1')$-lower-regular graph with $d_1' \simeps{\eps} d_1$. 
    By \Cref{thm:low reg_her}, for each $q'\ge q := C_0/ d_1'$, the number of sets of size $q'$ in $X_1'$ that are $(\eps_1,d_1')$-lower-regular with $X_2$ is at least $\left(1-\beta_0^{q'}\right) \binom{|X_1'|}{q'}$. Call those sets \emph{good}.

    Consider any $(\eps,d_2)$-lower-regular graph $G_2$ between $X_1$ and $X_3$. By \Cref{prop: reg_large_her}, and since $\eps_2 \le \eps_1$, $G_2'=G_2[X_1' \cup X_3]$ is an $(\eps_1,d_2')$-lower-regular graph with $d_2' \simeps{\eps} d_2$.
    By \Cref{prop:reg_degree}, at least $(1-\eps_1)n_3$ vertices in $X_3$ have degree at least
    \begin{equation}\label{1}
        (1-\eps_1)d_2'|X_1'| \ge(1-\eps_1)^2\eps_1 d_2 n_1 
        \ge (1-\eps_1)^2\eps_1 C / d_1 \ge (1-\eps_1)^3\eps_1 C / d_1' \ge C_0/d_1' = q
    \end{equation}
    to $X_1'$.
    Let $\mathcal{D}$ be the set of $n_3$-tuples of non-negative integers with at most $\eps_1n_3$ value less than $q$. We just showed that $(\deg_{G_2'}(v_1), \dots, \deg_{G_2'}(v_{n_3})) \in \mathcal{D}$.

    We say a bipartite graph between $X_1'$ and $X_3$ is bad if 
    \[
    |\{v \in X_3 \, : \, \Gamma_1(v) \text { is good }, |\Gamma_1(v)| \ge q \}|< (1-2\eps_1)n_3.
    \]
    Fix an integer $m'$ and consider a tuple $(r_1,r_2,\dots,r_{n_3}) \in \mathcal{D}$ such that $r_1+\dots+r_{n_3}=m'$.
    Among all bipartite graphs between $X_1'$ and $X_3$ with the vertices $1,2,\dots,n_3 $ in $X_3$ having degree $r_1,r_2,\dots,r_{n_3}$ respectively, the fraction of bad graphs is at most 
    \begin{align*}
    2^{n_3}\left(\beta_0^{q}\right)^{\eps_1 n_3} 
    \le 2^{n_3}\left(\beta_0^{(1-\eps_1)\eps_1 d_2 n_1}\right)^{\eps_1 n_3} 
    \le 2^{n_3}\beta_0^{(1-\eps_1)\eps_1^2 (d_2 n_1 n_3)} 
    \le \left(2\beta_0 ^{(1-\eps_1)\eps_1^2}\right)^{m_2} 
    \le \left(\frac{\beta}{2}\right)^{m_2}.   
    \end{align*}
    The calculation above is as follows. Since at most $\eps n_3$ vertices have degree at most $q$, there are at least $\eps_1 n_3$ vertices with degree at least $q\ge(1-\eps_1)\eps_1 d_2 n_1$ such that their neighbourhood in $X_1'$ is not a good set. There are at most $2^{n_3}$ ways to choose the set of these vertices.
    Secondly, for any vertex $v_i$, its neighbourhood is equally likely to be any set in $X_1$ with size $\Gamma_1(v_i)$, independently of all other vertices. Since the probability that $\Gamma_1(v_i)$ is bad is at most $\beta_0^{\deg(v_i)}\le \beta_0^{q}$, the probability that $\Gamma_1(v_i)$ are bad for all $\eps n_3$ vertices is at most $\left(\beta_0^{q}\right)^{\eps_1 n_3}$. Also noted that $n_1 \ge C (d_1d_2)^{-1}$ imply $m_2= d_2n_1n_3 \ge (C/d_1) n_3 \ge n_3$. 
    
    Each bipartite graph between $X_1'$ and $X_3$ with exactly $m'$ vertices can be extended to a bipartite graph between $X_1$ and $X_3$ with exactly $m$ edges in exactly $\binom{(n_1-|X_1'|)n_3}{m_2-m'}$ ways.
    Hence, The number of $(\eps,d_2)$-lower-regular graph $G_2$ between $X_1$ and $X_3$ such that the subgraph $G_2'$ is bad is at most 
    \begin{align*}
    &\sum_{0 \le m' \le m} \enspace \sum _{r_1+\dots+r_{n_3}=m'} 
    |\{G_2 \; : \; G_2' \text{ is bad }, \forall i, \deg_{G_2'}(v_i)=r_i\}|{(n_1-|X_1'|)n_3 \choose m_2-m'} \\
    &= \sum_{0 \le m' \le m} \enspace \sum _{\substack{r_1+\dots+r_{n_3}=m' \\ (r_1,\dots,r_{n_3}) \in \mathcal{D}}}
    |\{G_2 \; : \; G_2'   \text{ is bad }, \forall i, \deg_{G_2'}(v_i)=r_i\}| {(n_1-|X_1'|)n_3 \choose m_2-m'}\\
    &\le\sum_{0 \le m' \le m} \enspace \sum _{\substack{r_1+\dots+r_{n_3}=m' \\ (r_1,\dots,r_{n_3}) \in \mathcal{D}}} \left(\frac{\beta}{2}\right)^{m_2}|\{G_2 \; : \;\forall i, \deg_{G_2'}(v_i)=r_i\}| {(n_1-|X_1'|)n_3 \choose m_2-m'}\\
    &\le \left(\frac{\beta}{2}\right)^{m_2}  \sum_{0 \le m' \le m} {(n_1-|X_1'|)n_3 \choose m_2-m'} \sum _{r_1+\dots+r_{n_3}=m'} |\{G_2 \; : \; \forall i, \deg_{G_2'}(v_i)=r_i\}| \\
    &= \left(\frac{\beta}{2}\right)^{m_2}  \sum_{0 \le m' \le m} {(n_1-|X_1'|)n_3 \choose m_2-m'} |\{G_2 \; : \; G_2' \text{ has exactly $m'$ edges}\}|
    \le \left(\frac{\beta}{2}\right)^{m_2}\binom{n_1 n_3}{m_2}
    \end{align*}
    The last inequality follows since we can count the number of bipartite graphs between $X_1$ and $X_3$ with exactly $m$ edges by choosing exactly $m'$ edges between $X_1'$ and $X_3$ first, then choosing the remaining $m_2-m'$ edges between $X_1 \setminus X_1'$ and $X_3$.
    
    Apart from those $(\beta/2)^{m_2}\binom{n_1 n_3}{m_2}$ graphs, the induced graph in $[X_1'\cup X_3]$ have at least $(1-2\eps)n_3$ vertices $v \in X_3$ such that $\Gamma_1(v)$ is good, which mean that $\Gamma_1(v)$ is $(\eps_1,d_1')$-lower-regular with $X_2$.
    By \Cref{prop:reg_degree}, for each such vertex $v \in X_3$, at least $(1-\eps_1)n_2$ vertices $x_2\in X_2$ have degree at least 
    \[
    (1-\eps_1)d_1'|\Gamma_1(v)|\ge (1-\eps_1)^3 d_1d_2|X_1'| \ge (1-3\eps_1) d_1d_2|X_1'|
    \]
    into $\Gamma_1(v)$.
    In other words, there are at least $(1-\eps_1)n_2$ vertices in $Y$ with one endpoint as $v$ having a degree at least $(1-3\eps_1) d_1d_2|X_1'|$ into $X_1'$. 
    In total, there are at least $\left((1-2\eps_1)n_3\right)\left((1-\eps_1)n_2\right) \ge (1-3\eps_1)n_2n_3 = (1-3\eps_1)|Y|$ such vertices in $Y$, as required.
\end{proof}

Back to the proof of \Cref{lem:Auxiliary_regular}. Using the claim for all set $|X_1'|$ of size at least $\eps_1|X_1'|$, we have that, apart from at most $2^{n_1}(\beta/2)^{m_2}\binom{n_1 n_3}{m_2}<\beta^{m_2}\binom{n_1 n_3}{d_2}$ $(\eps,d_2)$-lower-regular graphs between $X_1$ and $X_3$ with $m_2$ edges, any $X_1'\subset X_1$ and any $Y'\subset Y$ with $|X_1'|\ge \eps' |X_1|, |Y'|\ge \eps'|Y|$ have at least 
\begin{align*}
(|Y'|-3\eps_1|Y|)((1-3\eps_1)d_1d_2|X_1'|)
&\ge (1-3\eps_1/\eps')|Y'|(1-3\eps_1)d_1d_2|X_1'| \\
&\ge (1-\eps')d_1d_2|Y'||X_1'|      
\end{align*}
edges between them. That is, $X_2$ and $Y$ form a $(\eps',d_1d_2)$-lower-regular graph.
\end{proof}

Fix $n,m$ and let $d=m/n^2$. Let $X_2,X_3,X_4$ be vertex sets of size $n,dn,dn$ respectively. 
For $\delta>0$, we define $\B_\delta(3,n,dn,dn,d)$ to be the family of graphs $G \in \G(3,n,dn,dn,d)$ such that at least $\delta n$ vertices in $X_2$ are contained in at most $(1-\delta)n^2d^5$ canonical triangles in $G$. We will show that $\B_\delta(3,n,dn,dn,d)$ is small in $\G(3,n,dn,dn,d)$ with respect to $d_0(n,dn,dn)=n^{-1/m_2(K_4)}=n^{-2/5}$. To do this, we first extend an empty graph with the edges between $X_2$ and $X_3$, then the edges between $X_3$ and $X_4$, and then the edges between $X_2$ and $X_4$. Hence, we let $H_0$ be the empty graph on vertices $1,2,3$ and define $H_{1,1}=H_0+(1,2)$ (that is, adding the edge $(1,2)$ to $H_0$), $H_{2,1}=H_{1,1}+(1,3)$, and $H_{2,2}=H_{2,1}+(2,3)=K_3$.
For all $\eps>0$ and all $3\times 3$ symmetric matrix $D$, define $\B_{\eps}(H_{2,1},n_1,n_2,n_3,D)$ as the family of graphs $G$ in $\G(H_{2,1},n_1,n_2,n_3,D)$ such that the Aux-path graph $A(G)$ between $X_1$ and $Y=X_2 \times X_3$ is not $(\eps,D_{1,2}D_{1,3})$-lower-regular.

\begin{corollary}\label{lem: Aux graph reg}
For $0 < \beta,\eps' < 1$, there exist constants $\eps,C>0$ such that for all $n_1,n_2,n_3, D$, and $m_{1,3}:=D_{1,3}n_1n_3$ satisfying 
\[
n_1\ge \frac{C}{D_{1,2}D_{1,3}} \quad, \enspace 
n_2 \ge \frac{C}{D_{1,2}} \quad, \enspace 
n_3 \ge \frac{C}{D_{1,3}}.
\]
every $G \in \G(H_{1,1},n_1,n_2,n_3,D,\eps)$ can be extended to $\B_{\eps'}(H_{2,1},n_1,n_2,n_3,D)$ in at most
\[
\beta ^{m_{1,3}}\binom{n_1 n_3}{m_{1,3}} 
\]
ways.
\end{corollary}
\begin{proof}
Follows immediately from \Cref{lem:Auxiliary_regular} since a regular graph is also lower regular.
\end{proof}

For all $\delta>0$ and all $3\times 3$ symmetric matrix $D$, define $\B_\delta(3,n_1,n_2,n_3,D)$ to be the family of graphs $G$ in $\G(3,n_1,n_2,n_3,D)$ with vertex sets $X_1,X_2,X_3$ such that at least $\delta n_1$ vertices in $X_1$ are contained in at most $(1-\delta)n_2 n_3 D_{1,2}D_{1,3}D_{2,3}$ canonical triangles in $G$.

\begin{lemma}\label{lem: right triangles}
For $0 < \beta,\delta< 1$, there exist constants $\eps_0,C>0$ such that for all $n_1,n_2,n_3,D$, $m_{2,3}:=D_{2,3}n_2 n_3$, and $0<\eps \le \eps_0 $ satisfying 
\[
 m_{2,3}\ge \frac{C}{D_{1,2}D_{1,3}} \enspace, \quad
\]
every $G \in \G(H_{2,1},n_1,n_2,n_3,D,\eps)\setminus \B_{\eps}(H_{2,1},n_1,n_2,n_3,D)$ can be extended to $\B_{\delta}(3,n_1,n_2,n_3,D)$ in at most 
\[
\beta ^{m_{2,3}}\binom{n_2 n_3}{m_{2,3}}    
\]
ways.
\end{lemma}

\begin{proof}
Apply \Cref{thm:low reg_her} with $\beta \leftarrow \beta$, $\eps' \leftarrow \delta$, $V_1 \leftarrow X_1$, $V_2 \leftarrow Y:= X_2 \times X_3$ to obtain constants $\eps \leftarrow \eps$ and $C \leftarrow C $. 
Consider any $\eps\le\eps_0$, and any $G \in \G(H_{2,1},n_1,n_2,n_3,D,\eps)\setminus \B_{\eps}(H_{2,1},n_1,n_2,n_3,D)$.
Each bipartite graph between $X_2$ and $X_3$ can be represented by a subset of $Y$ containing all of its edges. 
Since $m_{2,3}\ge C(D_{1,2}D_{1,3})^{-1}$, at most $\beta^{m_{2,3}}\binom{n_2 n_3}{m_{2,3}}$ bipartite graphs with exactly $m_{2,3}$ edges between $X_2$ and $X_3$ is not $(\delta,D_{1,2}D_{1,3})$-lower-regular with $X_1$.
It is enough to show that for all $(\eps)$-regular graphs $G_{2,3}$ that do not belong to the exception above, we have $G \cup G_{2,3} \notin \B_{\delta}(3,n_1,n_2,n_3,D)$. 

Since $X_1$ is $(\delta,D_{1,2}D_{1,3})$-lower-regular with $G_{2,3}$, at least $(1-\delta)n_1$ vertices $v_1 \in X_1$ have degree at least $(1-\delta)D_{1,2}D_{1,3}|G_{2,3}|$ in $G_{2,3} \seq Y$ by \Cref{prop:reg_degree}.
Hence, such a vertex $v_1$ belongs to at least 
\[
(1-\delta)D_{1,2}D_{1,3}|G_{2,3}|=(1-\delta)D_{1,2}D_{1,3}m_{2,3}
=(1-\delta)n_2 n_3 D_{1,2}D_{1,3}D_{2,3}
\]
canonical triangles in $G \cup G_{2,3}$. Therefore, at most $\delta n_1$ vertices which fail to do so, and hence $G \cup G_{2,3} \notin \B_{\delta}(3,n_1,n_2,n_3,D)$.
\end{proof}

Note that the edges between $X_2$ and $X_3$ don't need to be lower-regular, only a $\beta ^{m_{2,3}}$ fraction of all $\binom{n_2 n_3}{m_{2,3}}$ possible bipartite graphs with $m_{2,3}$ edges don't have the required property.

\begin{lemma}\label{lem:small family}
For all $\delta>0$, $\B_{\delta}(3,n,dn,dn,d)$ is small in $\G(3,n,dn,dn,d)$ with respected to $d_0(n,dn,dn)=n^{-2/5}$.    
\end{lemma}
\begin{proof}
Fix $\beta>0$.
Let $H_0$ be the empty graph on vertices $1,2,3$. Define $H_{1,1}=H_0+(1,2)$, $H_{2,1}=H_{1,1}+(1,3)$, and $H_{2,2}=H_{2,1}+(2,3)=K_3$ (same as before). Note that $\mathbb H: H_0 \subset H_{1,1} \subset H_{2,1} \subset H_{2,2}$ is a valid sequence for $H$, and that $\B_{\delta}(3,n,dn,dn,d)$ is the same as $\B(H_{2,2},n,dn,dn,d)$.

Applying \Cref{lem: right triangles} with $\beta \leftarrow \beta, \delta \leftarrow \delta$ yields constants $\eps_0 \leftarrow \eps_0, C_1 \leftarrow C$.
Applying \Cref{lem: Aux graph reg} with $\beta \leftarrow \beta, \eps' \leftarrow \eps_0$ yields constants $\eps_\beta \leftarrow \eps_0, C_2 \leftarrow C$.
Let $C_\beta=\max \{C_1,C_2,1\}$. For any $d \ge C_\beta d_0=  C_\beta n^{-2/5}$ , we have 
\[
m_{2,3}=D_{2,3}n_2n_3 = d(dn)(dn) = (d^5n^2)/d^2 \ge C_\beta^5/d^2 \ge C_1(D_{1,2}D_{1,3})^{-1},
\]
so the hypothesis of \Cref{lem: right triangles} is satisfied. Furthermore,
\[
n_1=n = d^2n/d^2 \ge C_\beta^2n^{1/5}/d^2 \ge C_2/(D_{1,2}D_{1,3}) \quad , 
\enspace
n_2=dn= d^2n/d^2 \ge C_\beta^2n^{1/5}/d^2 \ge C_2/D_{1,2},
\]
and similarly for $n_3$. So the hypothesis of \Cref{lem: Aux graph reg} is also satisfied. Define 
\[
\B(H_{1,1},n,dn,dn,d)= \emptyset \quad , \enspace
\B(H_{2,1},n,dn,dn,d)=\B_{\eps_1}(H_{2,1},n,dn,dn,d).
\]
Condition \ref{def:gen_small i} of \Cref{def:gen_small} holds as $|\B(H_{1,1},n,dn,dn,d)|=0$. By \Cref{lem: Aux graph reg}, condition \ref{def:gen_small ii} holds for $k=2, i=1$. By \Cref{lem: right triangles}, condition \ref{def:gen_small ii} holds for $k=2, i=2$.
\end{proof}

We can finally prove our main theorem.
\begin{theorem}
    $H=K_4$ case of \Cref{conj: count} is true. 
\end{theorem}

\begin{proof} 
Fix any $\beta,\delta>0$ and let $\delta'$ be small enough so that 
\[
(3\delta')^{\delta} \le \frac{\beta}{8} \quad , \enspace 
\delta' \le \min\{ \frac{\delta}{12}, \eps, 1-\eps \}.
\]
Let the 4 partition classes be $V_1, V_2, V_3, V_4$. 
Apply \Cref{lem:gen_small_bad} with $\ell \leftarrow 4, \beta \leftarrow \beta/2, \eps' \leftarrow \delta'$ to obtain constant $C\leftarrow C$ and $\eps \leftarrow \eps$. We may assume $C\ge 2$. By \Cref{lem:small family}, $\B_{\delta'}(3,n,dn,dn,d)$ is small in $\G(3,n,dn,dn,d)$ with respected to $d_0=n^{-2/5}$.

Recall that in this case, $m_2(H)=5/2$. Let $H'$ be a graph obtained by deleting the edge $(1,2)$ from $H$, $x=(1-\delta')m/n$ and $d'=(1-\delta')m/n^2$. As $m/n^2 \ge C n^{-1/m_2(H)} = C n^{-2/5}$, we have
\[
\begin{split}
x=(1-\delta')\frac{m}{n} \quad, \enspace 
m\ge C n^{2-2/5}=C n^{8/5} \ge 2 n^{3/2}\sqrt{\log n} \quad, \enspace \\
\frac{m}{n^2} \ge C n^{-2/5}=C d_0 \ge 1/x \quad, \enspace
(1 - \eps)\frac{m}{n^2} \ge d'=(1-\delta')\frac{m}{n^2} \ge  \eps\frac{m}{n^2}.
\end{split}
\]
As all the conditions of \Cref{lem:gen_small_bad} are met, we can apply it with the family $\B_{\delta'}(3,n,nd',nd',d')$.
Hence, all but $(\beta/2)^{m} {n^2 \choose m}^5$ graph $G \in \G(H',n,m,\eps)$ is a \emph{good graph}, which we define as at least $(1-\delta')n$ vertices $v_1 \in V_1$ have a property that $G(v_1):=V_2 \cup \left(\Gamma(v_1) \cap (V_3 \cup V_4) \right)$ contained a member of family $\G(3,n,dn,dn,d) \setminus \B_{\delta'}(3,n,dn,dn,d)$. Called all such vertices $v_1$ good vertices.

We call the an edge between $V_1$ and $V_2$ \emph{good} edge if $\deg_{H'}(e,G) \ge (1-6\delta')n^2d^5$.
Fix a good vertex $v_1\in V_1$. By definition of $\B_{\delta'}(3,n,nd',nd',d')$, at least $(1-\delta')n$ vertices $v_2 \in V_2$ is such that $v_2$ is contained in at least
\[
(1-\delta')n^2(d')^5 = (1-\delta')^6n^2d^5 \ge (1-6\delta')n^2d^5
\] 
canonicals triangles in $G(v_1)$.
Hence, at least $(1-\delta')n$ edges $e$ between $V_1$ and $V_2$ with $v_1$ as one endpoint are good edges.
In total, at least $[(1-\delta')n][(1-\delta')n]\ge (1-3\delta')n^2$ edges between $V_1$ and $V_2$ are good.
If a collection of edges $G_1$ between $V_1$ and $V_2$ contain at least  $(1-\delta/2)m$ good edges, then the number of canonical copies of $K_4$ contained in $G \cup G_1$ is at least
\[
(1-\delta/2)m \cdot (1-6\delta')n^2d^5 \ge (1-\delta) n^4d^6
\ge (1-\delta)n^{|V(K_4)|}\left(\frac{m}{n^2}\right)^{|E(K_4)|}.
\]

The number of ways to choose at least $\delta m$ bad edges from $m$ edges is at most 
\begin{align*}
{3\delta'n^2 \choose \delta m}{n^2 \choose m-\delta m} 
&\leBy{eq:binomial coefficient 1} \left(3\delta'\right)^{\delta m}{n^2 \choose \delta m}{n^2 \choose m-\delta m} \\
&\leBy{eq:binomial coefficient 2} \left(4(3\delta')^{\delta}\right)^m
{n^2 \choose m} 
\le \left(\frac{\beta}{2}\right)^m {n^2 \choose m}
\end{align*}

Therefore, for each good graph $G \in \G(H',n,m,\eps)$, there is at most $(\beta/2)^m{n^2 \choose m}$ extensions to $G' \in \G(4,n,m,\eps)$ such that $G'$ contain at most $(1-\delta) n^4 d^6$ canonical copies of $K_4$. In total, the number of graphs in $G' \in \G(4,n,m,\eps)$ that  contain at most $(1-\delta)n^{|V(K_4)|}\left(\frac{m}{n^2}\right)^{|E(K_4)|}$ canonical copies of $K_4$ is at most 
\[
\left(\frac{\beta}{2}\right)^m{n^2 \choose m}^6+{n^2 \choose m}^5 \left(\left(\frac{\beta}{2}\right)^m {n^2 \choose m} \right)
\le \beta^m {n^2 \choose m}^6,
\]
as required.
\end{proof}

\section{Proof of \Cref{lem:gen_small_bad}}\label{section:lemma_proof}

We first define the family $\GG(H, n_1,\dots,n_\ell, D, \eps)$, a superset of $\G(H, n_1,\dots,n_\ell, D, \eps)$, as follows. $\GG(H, n_1,\dots,n_\ell, D, \eps)$ is the family of $\ell$-partite graphs on $\ell$ pairwise disjoint vertex sets $V_1,\dots,V_\ell$ of size $n_i$ such that for all $1\le i< j\le \ell$, the bipartite graph between $V_i$ and $V_j$
\begin{itemize}
  \item[$(i)$]  has $m'_{i, j} \simeps{\eps} D_{i,j}n_i n_j$ edges and
  \item[$(ii)$] is $(\eps)$-regular.
\end{itemize}
Following the proof in \cite{ProbCountLemma}, we prove the following lemma before proving \Cref{lem:gen_small_bad}. 

\begin{lemma}\label{l:tool2}
For $1\le i\le \ell$, let $x_i=x$ except possibly for one index $i'$, for which $x_{i'}=n$, and write $y_{a,b}=d'x_ax_b$.
Let $d_0(x_1,\dots,x_\ell)\ge 1/x$ be a monotone increasing function, and let $\B(\ell,x_1,\dots,x_\ell,d')$ be a family that is small with respect to $d_0$. Define $f(x)=(1+\eps')x$ and $f(n)=n$. 
For every $\beta>0$ and $\eps'>0$, there exist constants $\eps>0$ and $C$ such that the following holds. Let $x,m,n$ and $d'$ satisfy
\begin{equation}\label{eq: preconditions key lemma}
d:=\frac{m}{n^2}\ge C d_0(x_1,\dots,x_\ell), \quad
x \ge \frac{\log n}{d}, \quad
\eps' d \le  d' \le (1-\eps') d , \quad
\end{equation}

and assume $n$ sufficiently large, then all but at most $\beta^m{n^2\choose m}^{{\ell\choose 2}}$ graphs $G \in \G(\ell,n,m,\eps)$ satisfy the following property: 
the number of $\ell$-tuples $(X_1,\ldots,X_\ell)$ with $X_i \seq V_i$ and $|X_i| = f(x_i)$ for all $1 \leq i \leq \ell$, such that the induced graph $G[X_1,\dots,X_\ell]$ contains a member of the family
\[
\G(\ell,x_1,\dots,x_\ell,d',\eps') \setminus \B(\ell,x_1,\dots,x_\ell,d')
\]
is at least
is at least $(1 - \beta^x) \prod_{i=1}^\ell \binom{n}{f(x_i)}$. 
\end{lemma}

To simplify notation, we fix some parameters for the remainder of this section. First, let $\beta > 0$ and $\eps' > 0$ be the constants from \Cref{l:tool2},  and let $\B(\ell,x_1,\dots,x_\ell,d') \seq \G(\ell, x_1,\dots,x_\ell,d')$ be a family which is small with respect to $d_0(x_1,\dots,x_\ell)$.

Furthermore, we assume that $\alpha$ and $\hbeta$ satisfy
\begin{equation}\label{e3:tool2}
  (\ell\alpha)^{\eps'/\ell} < \frac{\beta}{32\ell}
  \enspace, \qquad
  \hbeta 
  \le \left(\frac{\beta}{\ell \cdot 2^{2\ell\alpha + \ell^2+3}} \right)^{k/(\alpha^2 \eps')}
  \enspace
\end{equation}
and 
$C_{\hbeta}$, $\eps_{\hbeta}$, and $n_{\hbeta}$ are defined as in $\Cref{def:gen_small}$.
For $1\le k \le \ell-1$ and $1\le i \le k$, let $e_{k,i},$ $\B(H_{k,i},x_1,\dots,x_\ell,d',\eps) \seq \G(H_{k,i},x_1,\dots,x_\ell,d',\eps)$ be defines as in \Cref{def:gen_small} (\ref{def:gen_small i}) and (\ref{def:gen_small ii}).
Moreover, let
\[
C := \frac{C_{\hbeta}}{\eps'}
\]
and let $x_1,\dots,x_\ell$, $n,m,d$ and $d'$ satisfy the conditions~(\ref{eq: preconditions key lemma}) of \Cref{l:tool2}. 
Observe that this implies 
\[
d'\ge \eps'd\ge \eps'Cd_0(x_1,\dots,x_\ell) 
\ge C_{\hbeta} d_0(x_1,\dots,x_\ell).
\]

We may assume that $x_1,\dots,x_\ell\geq n_{\hbeta}$. 

\begin{claim}\label{hl3:tool2}
Suppose $1\le k< \ell$ and $0<\eps\le \eps_{\hbeta}$ are fixed. 
Let $g(x)=\alpha n/x$ and $g(n)=1$.
Let $x'=\min_{1\le i \le k} x_i$, then there exist at most
\[
\frac{1}{\ell} \beta^m {n^2\choose m}^{{\ell \choose 2}}
\]
graphs in $\G(\ell,n,m)$ that satisfy the following properties:
\begin{itemize}
\item There exist $g(x')$ subgraphs $G_k^j \in \G(k, x_1,\dots,x_\ell,d',\eps) \setminus \B(k, x_1,\dots,x_\ell,d',\eps)$, $1\le j\le g(x')$, such that for all $1 \leq i \leq k$ and $1 \leq j \leq g(x')$, we have $|V(G_k^j) \cap V_i| = x_i$, and for $i$ such that $x_i \neq n$, the graphs $V(G_k^j) \cap V_i$ are pairwise disjoint for $1\le j\le g(x')$.
\item For all $1 \le j \le g(x')$, there exist pairwise disjoint sets $X_{k+1}^{j,s}\seq V_{k+1}$, $1 \le s \le g(x_{k+1})$, of size $x_{k+1}$ such that the graph induced by $V(G_k^j) \cup X_{k+1}^{j,s}$ 
contains an extension of $G_k^j$ that belongs to $\B(k+1, x_1,\dots,x_\ell,d', \eps)$.
\end{itemize}
\end{claim}

\begin{proof}
We prove the claim by constructing all graphs that satisfy the above properties.
Firstly, we select bipartite graphs with $m$ edges between $V_i$ and $V_{i'}$, for all pairs $1\le i < i' \le \ell$, $i' \neq k+1$. There are at most
\[
{n^2\choose m}^{{\ell\choose 2}-k}
\] 
ways to do that. Secondly, for all $1 \le j \le g(x')$, $1 \le s \le g(x_{k+1})$, we choose pairwise vertex disjoint set $X_i^j = V(G_k^j) \cap V_i$ (when $x_i \neq n)$ and the extension sets $X_{k+1}^{j,s}$. There are fewer than
\[
{n \choose x}^{k \alpha n/x + (\alpha n/x)^2} 
\le n^{\alpha kn + \alpha^2 n^2/x} \le 4^{(\ell \alpha n^2/x)\log n} 
\leBy{eq: preconditions key lemma} 2^{2\ell \alpha m}
\] 
ways to do that. Next, we choose the graphs $G_k^j$. For all $1 \leq i < i' \leq k$, one has to select $y_{i,i'}$ edges from the edges between the pair $(X_i^j, X_{i'}^j)$. Note that all the edges involved here are pairwise disjoint for all $j$, so the sum of the number of edges is at most $m$.
Hence, there are at most
\[
\begin{split}
  \prod_{1 \leq i < i' \leq k} \prod_{j = 1}^{g(x')}
    {|E(X_i^j, X_{i'}^j)| \choose y_{i,i'}}
  &\leBy{eq:binomial coefficient 3} \prod_{1 \leq i < i' \leq k}
    {\sum_{j = 1}^{g(x')} |E(X_i^j, X_{i'}^j)| 
    \choose g(x')y_{i,i'}} \\
  &\leq \prod_{1 \leq i < i' \leq k} {m \choose g(x')y_{i,i'}}
  \leq 2^{\ell^2 m}
\end{split}
\]
ways to choose the graphs $G_k^j$. Next, we choose bad extensions of $G_k^j$ to the sets $X_{k+1}^{j,s}$, and fill the remaining edges between $V_1,\dots ,V_k$ and $V_{k+1}$. 
At this point, we need the following claim, which we defer proving until later.

\begin{subclaim}\label{cl:bad-extensions}
There are at most
$\ell4^m\hbeta^{\alpha^2 \eps' m/k}\binom{n^2}{m}^k$
ways to make these choices.   
\end{subclaim}

Using Claim~\Cref{cl:bad-extensions}, we have shown that there are at most
\[
{n^2\choose m}^{{\ell\choose 2}-k} \cdot 2^{2\ell\alpha m}\cdot 2^{\ell^2 m}
\cdot \ell4^m\hbeta^{\alpha^2 \eps' m/k} \cdot \binom{n^2}{m}^k
\le (2^{2\ell\alpha + \ell^2+3} \hbeta^{\alpha^2 \eps'/k} )^m 
\binom{n^2}{m}^{\binom{l}{2}}
\le \left(\frac{\beta}{\ell}\right)^m\binom{n^2}{m}^{\binom{l}{2}}
\]
graphs in $\G(\ell,n,m)$ that satisfy the properties of \Cref{hl3:tool2}. 
This completes the proof of the main step.

\begin{proof}[Proof of subclaim~\ref{cl:bad-extensions}]
Recall that $v_{\sigma_k(i)}$ is the endpoint of $e_{k,i}$ that is not $v_{k+1}$, and that $H_{k,0}=K_k ,H_{k,i}=H_{k,i-1}+e_{k,i}$. 
For a fixed $j,s$, we are going to count the number of bad extensions $G_k^j$ to the sets $X_{k+1}^{j,s}$ by choosing the edges between $V_{\sigma_k(i)}\cap G_k^j$ and $X_{k+1}^{j,s}$ for $i=1,2, \dots,k$ in this order.
We denote $G(j,s,i')$ for the graph $G_k^j \cup X_{k+1}^{j,s}$ with edges between $\left(\cup_{i=1}^{i'} V_{\sigma_k(i)}\right) \cap G_k^j$ and $X_{k+1}^{j,s}$ already chosen.

For a fixed $i,j$, let $\B(i,j)$ be the set of $s$ such that $G(j,s,i'') \in \G(H_{k,{i''}},n_1,\dots,n_\ell,d,\eps) \setminus \B(H_{k,{i''}},n_1,\dots,n_\ell,d,\eps)$ for all $1 \le i'' \le i-1$
but $G(j,s,i) \in \B(H_{k,i},n_1,\dots,n_\ell,d,\eps)$. Let $\B(i)=\{(j,s) \; : \; s \in \B(i,j)\}$.
For all $g(x')g(x_{k+1})$ pairs of $(j,s)$, we have $G(j,s,0) \in \G(H_{k,0},n_1,\dots,n_\ell,d,\eps)\setminus \B(H_{k,0},n_1,\dots,n_\ell,d,\eps)$ and $G(j,s,k) \in \B(k+1,n_1,\dots,n_\ell,d,\eps)=\B(H_{k,k},n_1,\dots,n_\ell,d,\eps)$ (as we want all extension to be bad). Therefore, the sets $\B(i)$ partition the set of all pairs $(j,s)$, and note that $|\B(i)|=\sum_j |\B(i,j)|$. 
Since $\sum_i |\B(i)|=g(x')g(x_{k+1})$, by pigeon hole principles, there exists $i$ in $\{1,\dots,k\}$ such that $|\B(i)|\ge g(x')g(x_{k+1})/k$. So all bad extensions must have $|\B(i)|\ge g(x')g(x_{k+1})/k$ for some $i$.
It is enough to show that for each $i$, there's only at most $4^m\hbeta^{\alpha^2 \eps' m/k}\binom{n^2}{m}^k$ way to extend $G_k^j$ to the sets $X_{k+1}^{j,s}$ in which $|\B(i)|\ge g(x')g(x_{k+1})/k$. Then we can use the union bound to show that there are at most $\ell4^m\hbeta^{\alpha^2 \eps' m/k}\binom{n^2}{m}^k$ ways to extend to all bad extensions, which concludes \Cref{cl:bad-extensions}.

To do so, we consider 2 cases. 

\textbf{Case 1:} $x_{\sigma_k(i)}=n$. In this case, all other $x_{i'}$ have size $x$, and $y_{\sigma_k(i),k+1}=d'nx$. Since $|\B(i)|\ge g(x')g(x_{k+1})/k$, there is $j$ such that $|\B(i,j)|\ge g(x_{k+1})/k =\alpha n/(kx)$. Here we have at least $\alpha n/(kx)$ bad extensions, all of which are pairwise edge-disjoint. Therefore, by \Cref{def:gen_small}~\ref{def:gen_small ii}, the number of ways to choose such extensions is at most
\[
\prod_{s\in \B(i,j)} \left[\hbeta^{y_{\sigma_k(i),k+1}} \binom{x_{\sigma_k(i)}x_{k+1}}{y_{\sigma_k(i),k+1}}\right] 
\prod_{s\notin \B(i,j)} \binom{x_{\sigma_k(i)}x_{k+1}}{y_{\sigma_k(i),k+1}}
\le \hbeta^{|\B(i,j)|d'nx}\binom{nx}{d'nx}^{g(x_{k+1})}
\le \hbeta^{\alpha\eps'm/k}\binom{nx}{d'nx}^{\alpha n/x},
\]
as $|\B(i,j)|d'nx \ge \alpha d' n^2/k \geBy{eq: preconditions key lemma} 
\alpha \eps' dn^2/k= \alpha \eps'm/k$.
We now choose the remaining edges between $V_{\sigma_k(i)}$ and $V_{k+1}$. There are at most 
\[
\binom{n^2}{m-g(x_{k+1})y_{\sigma_k(i),k+1}}=\binom{n^2}{m-(\alpha n/x)d'nx}
\]
ways to do that. There are also at most $\binom{n^2}{m}^{k-1}$ way to choose edges between $V_{k+1}$ and each of $V_1,\dots, V_k$ that is not $V_{\sigma_k(i)}$. In total, the number of graphs such that $|\B(i)|\ge g(x')g(x_{k+1})/k$ is at most
\begin{align*}
\hbeta^{\alpha\eps'm/k}\binom{nx}{d'nx}^{\alpha n/x}\binom{n^2}{m-(\alpha n/x)d'nx}\binom{n^2}{m}^{k-1}
&\leBy{eq:binomial coefficient 3} 
\hbeta^{\alpha\eps'm/k}\binom{n^2}{\alpha d'n^2}\binom{n^2}{m-\alpha d'n^2} 
\binom{n^2}{m}^{k-1}\\
&\leBy{eq:binomial coefficient 2} 
4^m\hbeta^{\alpha \eps' m/k}\binom{n^2}{m}^{k}
\le 4^m\hbeta^{\alpha^2 \eps' m/k}\binom{n^2}{m}^{k},
\end{align*}
as required.

\textbf{Case 2:} $x_{\sigma_k(i)}=x$. In this case, all $|\B(i)|$ bad extensions are independent of each other, and so by \Cref{def:gen_small}~\ref{def:gen_small ii}, the number of such extensions is at most 
\[
\prod_{(j,s)\in \B(i)} \left[\hbeta^{y_{\sigma_k(i),k+1}} \binom{x_{\sigma_k(i)}x_{k+1}}{y_{\sigma_k(i),k+1}}\right] 
\prod_{(j,s)\notin \B(i)} \binom{x_{\sigma_k(i)}x_{k+1}}{y_{\sigma_k(i),k+1}}
\le \hbeta^{|\B(i)|y_{\sigma_k(i),k+1}}\binom{xx_{k+1}}{y_{\sigma_k(i),k+1}}^{g(x')g(x_{k+1})}.
\]
We now choose the remaining edges between $V_{\sigma_k(i)}$ and $V_{k+1}$. There are at most 
\[
\binom{n^2}{m-g(x')g(x_{k+1})y_{\sigma_k(i),k+1}}
\]
ways to do that. There are also at most $\binom{n^2}{m}^{k-1}$ way to choose edges between $V_{k+1}$ and each of $V_1,\dots, V_k$ that is not $V_{\sigma_k(i)}$. 
Using $g(x')g(x_{k+1})y_{\sigma_k(i),k+1}\ge \alpha^2 (n/x)(n/x_{k+1})d'xx_{k+1}\ge \alpha^2 n^2(\eps'd)=\alpha^2 \eps' m$, the number of graphs such that $|\B(i)|\ge g(x')g(x_{k+1})/k$ is at most
\begin{align*}
&\hbeta^{|\B(i)|y_{\sigma_k(i),k+1}}\binom{xx_{k+1}}{y_{\sigma_k(i),k+1}}^{g(x')g(x_{k+1})}\binom{n^2}{m-g(x')g(x_{k+1})y_{\sigma_k(i),k+1}}\binom{n^2}{m}^{k-1} \\   
&\leBy{eq:binomial coefficient 3} 
\hbeta^{(g(x')g(x_{k+1})/k)y_{\sigma_k(i),k+1}}\binom{g(x')g(x_{k+1})xx_{k+1}}{g(x')g(x_{k+1})y_{\sigma_k(i),k+1}}\binom{n^2}{m-g(x')g(x_{k+1})y_{\sigma_k(i),k+1}}\binom{n^2}{m}^{k-1}  \\
&\le 
\hbeta^{\alpha^2 \eps' m/k}\binom{n^2}{g(x')g(x_{k+1})y_{\sigma_k(i),k+1}}\binom{n^2}{m-g(x')g(x_{k+1})y_{\sigma_k(i),k+1}}\binom{n^2}{m}^{k-1}  \\
&\leBy{eq:binomial coefficient 3} 
4^m\hbeta^{\alpha^2 \eps' m/k}\binom{n^2}{m}^{k},
\end{align*}
as required.
\end{proof}

\end{proof}

\begin{claim}\label{chl3:tool2}
Suppose $0 < \eps \le \eps_{\hbeta}$ is fixed. Then all but at most
\[
\beta^m {n^2\choose m}^{{\ell \choose 2}}
\]
graphs in $\G(\ell,n,m)$ satisfy the following property.
For every $1\le i\le \ell$ there exist sets $W_i\subseteq V_i$ such that $W_i=\emptyset \text{ if } x_i=n$, and $|W_i|\le \alpha(\ell-1)n \text{ otherwise}$. Let $1\le k<\ell$, and suppose
\[
G_k \in
\G(k,x_1,\dots,x_\ell,d',\eps)
\setminus
\B(k,x_1,\dots,x_\ell,d',\eps)
\]
satisfies $|V(G_k)\cap V_i| = x_i$ and $V(G_k)\cap W_i=\emptyset$ for all $1\le i\le k$. Then there exists a set $W$ such that
\[
W=\emptyset \text{ if } x_{k+1}=n,
\qquad
W\subseteq V_{k+1} \text{ with } |W|=\alpha n \text{ otherwise},
\]
and the following holds. For every set
$X\subseteq V_{k+1}\setminus W$ with $|X|=x$, the graph induced by $V(G_k)\cup X$ contains no extension of $G_k$
belonging to $\B(k+1,x_1,\dots,x_\ell,d',\eps)$.
\end{claim}

\begin{proof}
We apply \Cref{hl3:tool2} for all $1\le k\le \ell-1$. This yields that all but at most
\[
\frac{\ell-1}\ell\beta^m{n^2\choose m}^{{\ell\choose 2}} 
\le \beta^m{n^2\choose m}^{{\ell\choose 2}}
\]
graphs in $\G(\ell, n, m)$ do not satisfy the properties stated in \Cref{hl3:tool2} for {\em any} $1\le k<\ell$. Now we show that all these graphs also satisfy the properties given in \Cref{chl3:tool2}.

Let $1\le k\le\ell-1$ and consider a maximum set of subgraphs $G_k^j \in \G(k, x_1,\dots,x_k,d', \eps) \setminus \B(k, x_1,\dots,x_k,d',\eps)$ such that for all $1 \leq i \leq k$, we have $|V(G_k^j) \cap V_i| = x_i$, $V(G_k^j) \cap V_i$ are pairwise disjoint for all $i$ such that $x_i \neq n$, and for all $j$, there exist pairwise vertex disjoint sets $X_{k+1}^{j,s}$, $1 \le s \le g(x_{k+1})$, of size $|X_{k+1}^{j,s}|=x_{k+1}$ such that $G[V(G_k^j) \cup X_{k+1}^{j,s}]$ contains an extension of $G_k^j$ that belongs to $\B(k+1, x_1,\dots,x_{k+1},d', \eps)$.

From \Cref{hl3:tool2} we immediately deduce that there can exist at most $g(x')$ such subgraphs $G_k^j$. That is, if we take the union of all these graphs for all $i$ such that $x_i \neq n$ and subsequently take the union over all $1\le k\le\ell-1$ of all their vertex sets, then we obtain sets $W_i$ of size at most $(\ell-1)\alpha n$ (and $W_i=\emptyset$ if $x_i=n$).

As we have chosen maximum sets of subgraphs $G_k^j$, we know that if we choose a $k$-tuple $(Z_1,\ldots,Z_k)$ with $Z_i \seq V_i \setminus W_i$ and $|Z_i| = x_i$ for all $1 \leq i \leq k$ that induces a graph containing $G_k \in \G(k, x_1,\dots,x_k,d', \eps) \setminus \B(k, x_1,\dots,x_k,d',\eps)$, there exists a set $W \subseteq V_{k+1}$ of size $|W|=\alpha n_{k+1}$ (or $0$ if $x_{k+1}=n$) such that for any set $Z \seq V_{k+1}\setminus W$ of size $x_{k+1}$ the induced graph $G[Z_1,\dots, Z_k, Z]$ contains no extension of $G_k$ that belongs to $\B(k+1, x_1,\dots,x_{k+1},d', \eps)$.
\end{proof}

\begin{claim}\label{hl4:tool2}
Let $V$ be a set of size $n$ and let $W\subseteq V$ be a set of size
$|W|\le \ell\alpha |V|$. Then the number of sets
$\hX \seq V$ of size $|\hX|=(1 + \eps')x$ that contain at least $\frac{1}{\ell}\eps' x$ vertices from $W$ is at most
\[
\left(\frac\beta{2\ell}\right)^x \binom{n}{(1+ \eps')x}.
\]
\end{claim}

\begin{proof}
Using~(\ref{eq:binomial coefficient 1}) and~(\ref{eq:binomial coefficient 2}), we deduce that
\begin{equation*}
\begin{split}
  &\bigg|\left \{ \hX \subseteq V: |\hX|=(1+ \eps')x
    \wedge |\hX \cap W|\ge \frac{1}{\ell}\eps'x \right\}\bigg| \\
  &\leq {|W| \choose \frac{1}{\ell}\eps'x} 
     {|V|\choose (1 + \frac{\ell - 1}{\ell}\eps')x} \\
  &\leq \left[4^{1+\eps'} \cdot
    (\ell\alpha)^{\frac{1}{\ell}\eps'}\right]^x 
    \cdot {|V|\choose (1 + \eps')x} 
  \;\leBy{e3:tool2}\; \left(\frac\beta{2\ell}\right)^x 
    \cdot {|V|\choose (1 + \eps')x}
\end{split}
\end{equation*}
as claimed.
\end{proof}

\begin{proof}[Proof of \Cref{l:tool2}]

First, we specify how to choose $\eps$. Let $\eps'' := \min\{\eps_{\hbeta}, \, \eps'/\ell\}$. We repeatedly invoke \Cref{thm:reg_her}. Applying \Cref{thm:reg_her} with $\beta \leftarrow \beta / (2\ell^2)$ and $\eps' \leftarrow \eps'' / 6$ yields constant $\eps_1 \leftarrow \eps$. We apply \Cref{thm:reg_her} again with parameters $\beta \leftarrow \beta / (2\ell^2)$ and $\eps' \leftarrow \eps_1 / 3$, and obtain $\eps$. We may assume that $\eps' \geq \eps'' \geq 2\eps_1 \geq 2\eps$. 

Since we have $\eps \leq \eps_{\hbeta}$, \Cref{chl3:tool2} asserts that for all but $\beta^m{n^2 \choose m}^{{\ell\choose 2}}$ of the graphs in $\G(\ell,n,m,\eps)$, there exist suitable sets $W_i\subseteq V_i$ for all $1 \leq i \leq \ell$. Suppose $G \in \G(\ell,n,m,\eps)$ is such a graph. We shall show that, in fact, $G$ satisfies the properties stated in \Cref{l:tool2}, that is, most $\ell$-tuples $(X_1,\dots,X_\ell)$ of size $|X_i|=f(x_i)$ in $G$ are good in the following sense. For $k \ge 1$, we call $(X_1,\dots,X_k)$ a {\em good} $k$-tuple if it satisfies the following properties:
\begin{itemize}
	\item For all $1 \le i \le k$, we have $X_i \subseteq V_i$ and $|X_i|=f(x_i)$, and there exists a subset $X_i' \seq X_i \setminus W_i$ of size $x_i$ such that $X'_i$ forms an $(\eps_1)$-regular pair of density $d'_{i,j} \simeps{\eps_1} d$ with each set $V_j$, $k < j \leq \ell$ ($d'_{i,j}$ here are not necessary the same as $d'$).
	\item The induced graph $G[X_1',\dots,X_k']$ contains a member of the family $\G(k, x_1,\dots,x_k,d', \eps'') \setminus \B(k, x_1,\dots,x_k,d',\eps'')$ as a subgraph.
\end{itemize}
Note that a good $1$-tuple is just a set $X_1 \seq V_1$ of size $f(x_1)$ that contains a subset $X_1' \seq V_1 \setminus W_1$ of size $|X_1'| = x_1$ which forms an $(\eps_1)$-regular pair of density $d'_{1, j} \simeps{\eps_1} d$ with each set $V_j$, $2 \leq j \leq \ell$. Formally, we also introduce the unique $0$-tuple, which is always good.

We shall show by induction on $k$, $0 \leq k \leq \ell$, that $G$ contains at least
\begin{equation}\label{e5:tool2}
\prod_{i=1}^k \left((1 - \beta^x / \ell){n\choose f(x_i)}\right)
\end{equation}
good $k$-tuples. In particular, for $k = \ell$, there are at least
\[
  \prod_{i=1}^\ell \left((1 - \beta^x / \ell) \binom{n}{f(x_i)}\right)
  \ge (1-\beta^x) \prod_{i=1}^\ell \binom{n}{f(x_i)}
\]
$\ell$-tuples in $G$ inducing a graph that contains a member of the family $\G(k, x_1,\dots,x_k,d', \eps'') \setminus \B(k, x_1,\dots,x_k,d',\eps'')$ and thus of $\G(k, x_1,\dots,x_k,d', \eps') \setminus \B(k, x_1,\dots,x_k,d')$  as a subgraph. 
Hence, it remains to prove~\eqref{e5:tool2}. 

The base case, when $k = 0$, clearly holds. Hence, suppose~\eqref{e5:tool2} holds for $k$, $0 \leq k < \ell$. We shall show that it also holds for $k + 1$. Consider any good $k$-tuple $(X_1,\dots,X_k)$ and let $G_k$ denote the subgraph of $G[X_1',\dots,X_k']$ that is a member of the family $\G(k, x_1,\dots,x_k,d', \eps'') \setminus \B(k, x_1,\dots,x_k,d',\eps'')$. If $k = 0$, $G_k$ is just the empty graph. We shall show that one can extend this tuple so that it forms a good $(k+1)$-tuple in at least
\[
(1-\beta^x /\ell)\cdot \binom{n}{f(x_{k+1})}
\]
many ways. Clearly, this completes the proof of~\eqref{e5:tool2}.

First assume $x_{k+1}=n$. In this case $W=W_{k+1}=\emptyset$ and we can take $X'_{k+1}=X_{k+1}=V_{k+1}$. The condition that $X'_{k+1}$ forms an $(\eps_1)$-regular pair of density $d$ with each set $V_j$, $k+1 < j \leq \ell$ is inherits from $G$.
Since $\eps_1 \le \eps''/2$ and by induction step, $X'_{k+1}$ forms an $(\eps''/2)$-regular pair of density $d'$ with each set $V_j$, $j<k$. Hence, the induced graph $G[X_1',\dots,X_k']$ contains a member of the family $\GG(k, x_1,\dots,x_k,d, \eps''/2)$ as a subgraph. By \Cref{chl3:tool2} and the choice of $X'_{k+1} \seq V_{k+1} \setminus (W_{k + 1} \cup W)$, this induced graph $G[X_1',\dots,X_k']$ contains no member of the family $\B(k, x_1,\dots,x_k,d',\eps'')$ as a subgraph. 
It remains to provide any subgraph of $G[X_1',\dots,X_{k+1}']$ that is a member of the family $\G(k+1,x_1,\dots,x_{k+1},d',\eps'')$. This follows by applying \Cref{lem:sub_of_eps_is_eps} to each pair $(X'_{k+1}, X'_j)$, $1 \leq j \leq k$.

Now we can assume $x_{k+1}=x$. Observe that by \Cref{thm:reg_her} and the choice of $\eps$, we know that all but at most
\begin{equation}\label{eq:good sets 1}
(\ell - k - 1)\left(\frac{\beta}{2\ell^2}\right)^{(1 + \eps')x}\binom{n}{(1 + \eps')x} 
\end{equation}
of the subsets of $V_{k + 1}$ of size $(1 + \eps')x$ contain a family of subsets $(\tX_{k + 1}^j)_{k + 1 < j \leq \ell}$, each of which has size at least
\[
  (1 - \eps_1/3)(1 + \eps') x \geByM{\ell \eps_1 \leq \eps'} \left(1 + \frac{\ell - 1}{\ell}\eps'\right)x
\]
and $\tX_{k+1}^j$ forms an $(\eps_1/ 3)$-regular pair of density $d(\tX_{k+1}^j, V_j) \simeps{\eps} d$ with $V_j$. Employing \Cref{thm:reg_her} again, we have by the choice of $\eps_1$ that all but at most
\begin{equation}\label{eq:good sets 2}
k\left(\frac{\beta}{2\ell^2}\right)^{(1 + \eps')x}{n \choose (1 + \eps')x}
\end{equation}
of the subsets of $V_{k+1}$ of size $(1 + \eps')x$ contain a family of subsets $(\tX_{k+1}^j)_{1 \leq j \leq k}$, each of which has size at least
\[
  (1 - \eps''/6)(1 + \eps') x 
  \geByM{\ell \eps'' \leq \eps'} \left(1 + \frac{\ell - 1}{\ell}\eps'\right) x
\] 
and $\tX_{k+1}^j$ forms an $(\eps''/6)$-regular pair of density $d(\tX_{k+1}^j, X_j') \simeps{3\eps_1} d$ with $X_j'$. The bound on $d(\tX_{k + 1}^j, X_j')$ follows from
\begin{equation*}
\begin{split}
  (1 - 3\eps_1)d
  &\leq (1 - \eps_1)^2d \\
  &\leq (1 - \eps_1)d(V_{k+1}, X'_j) \\
  &\leq d(\tX_{k+1}^j, X'_j) \\
  &\leq (1 + \eps_1)d(V_{k+1}, X'_j) \\
  &\leq (1 + \eps_1)^2d
  \leq (1 + 3\eps_1)d .
\end{split}
\end{equation*}

Moreover, if $k \geq 1$, then by \Cref{chl3:tool2} there exists a set $W \subseteq V_{k+1}$ of size at most $\alpha n$ such that the graph induced by any subset of $V_{k+1}\setminus W$ of size $x$ and the sets $X'_1,\dots,X'_k$ is not a member of $\B(k+1,x_1,\dots,x_{k+1},d',\eps)$. For $k = 0$, we pick $W$ as the empty set. Hence, \Cref{hl4:tool2} yields that at least
\begin{equation}\label{eq:good sets 3}
\left(1 - \left(\frac{\beta}{2\ell}\right)^x\right) {n\choose (1 + \eps')x}
\end{equation}
subsets $X_{k+1}$ of $V_{k+1}$ contain a subset $\hX_{k+1}$ of size at least $\left(1 + \frac{\ell - 1}{\ell}\eps'\right)x$ that is disjoint from $W_{k+1}$ and $W$.

Combining~\eqref{eq:good sets 1}, \eqref{eq:good sets 2}, and \eqref{eq:good sets 3}, we obtain that at least
\[
	\left(1- 2(\ell-1)\left(\frac{\beta}{2\ell^2}\right)^{(1+\eps')x}
		-\left(\frac{\beta}{2\ell}\right)^{x}\right) {n\choose (1+\eps')x}
	\geq \left(1- \beta^x / \ell\right){n\choose (1+\eps')x}
\]
sets $X_{k+1}\subseteq V_{k+1}$, $|X_{k+1}|=(1+\eps')x$, contain a family of subsets $(\tX_{k+1}^j)_{j \in [\ell] \setminus \{k+1\}}$, such that $\tX^j_{k+1}$ forms a regular pair with $X_j$ for $j \leq k$ and with $V_j$ for $j > k + 1$ respectively, as well as a subset $\hX_{k+1} \seq X_{k+1} \setminus (W_{k+1} \cup W)$. Since each of those subsets has size at least $\left(1 + \frac{\ell - 1}{\ell}\eps'\right)x$, we conclude that the set
\[
  \tX_{k+1} 
  := \hX_{k+1} \cap \bigcap_{j \in [\ell] \setminus \{k+1\}} \tX_{k+1}^j
\]
has size at least $x$. Now consider any subset $X'_{k+1} \seq \tX_{k+1}$ of cardinality exactly $x$. Since we have $x \geq |\tX_{k+1}^j| / 2$ for all $j \in [\ell] \setminus \{k + 1\}$, according to \Cref{prop: reg_large_her} any such set forms an $(\eps_1)$-regular pair of density $d(X'_{k + 1}, V_j) \simeps{\eps_1} d$ with $V_j$ for all $k + 1 < j \leq \ell$. This is deduced as follows:
\begin{equation*}\label{eq: bound on d(X_k+1, V_j)}
\begin{split}
  (1 - \eps_1)d
  &\leq (1 - \eps_1 / 3)(1 - \eps)d \\
  &\leq (1 - \eps_1 / 3)d(\tX_{k + 1}^j, V_j) \\
  &\leq d(X'_{k + 1}, V_j) \\
  &\leq (1 + \eps_1 / 3)d(\tX_{k+1}^j, V_j) \\
  &\leq (1 + \eps_1 / 3)(1 + \eps)d
  \leq (1 + \eps_1)d
  .
\end{split}
\end{equation*}
Moreover, regularity is also inherited by $X_{k+1}'$ w.r.t. $X'_j$, $1 \leq j \leq k$, again by~\Cref{prop: reg_large_her}. That is, $X_{k+1}'$ forms an $(\eps''/2)$-regular pair of density $d(X'_{k+1}, X'_j) \simeps{\eps''/2} d $ with $X'_j$ for all $1\le j\le k$. This follows from
\begin{equation*}
\begin{split}
  (1 - \eps'' / 2)d
  &\leq (1 - \eps'' / 6)(1 - 3\eps_1)d \\
  &\leq (1 - \eps'' / 6)d(\tX^j_{k+1}, X'_j) \\
  &\leq d(X'_{k+1}, X'_j) \\
  &\leq (1 + \eps'' / 6)d(\tX_{k+1}, X'_j) \\
  &\leq (1 + \eps'' / 6)(1 + 3\eps_1)d
  \leq (1 + \eps'' / 2)d
  .
\end{split}
\end{equation*}

We claim that all those sets $X_{k+1}$ form a good $(k+1)$-tuple $(X_1,\dots,X_k,X_{k+1})$. This immediately follows from the construction for $k = 0$. Thus, suppose that $k \geq 1$. The induction hypothesis tells us that $G[X'_1,\ldots, X'_k]$ contains a member of the family $\GG(k, x_1,\dots,x_k,d, \eps'' / 2)$. As $d(X'_{k+1}, X'_j) \simeps{\eps''/2} d$, and by construction of $X'_{k+1}$, for all $1 \leq j \leq k$, we have
\[
  |E(X'_{k+1}, X'_j)| \simeps{\eps''/2} x_jx_{k+1} d.
\]
Hence, $G[X'_1,\ldots, X'_{k+1}]$ is a member of the family $\GG(k+1, x_1,\dots,x_{k+1},d, \eps'' / 2)$. And since $G[X'_1,\ldots, X'_{k+1}]$ contains no extension of $G_k$ that belongs to $\B(k+1, x_1,\dots,x_{k+1},d', \eps'')$ owing to \Cref{chl3:tool2} and the choice of $X'_{k+1} \seq V_{k+1} \setminus (W_{k + 1} \cup W)$, it remains to provide any subgraph of $G[X_1',\dots,X_{k+1}']$ that is a member of the family $\G(k+1,x_1,\dots,x_{k+1},d',\eps'')$. This follows by applying \Cref{lem:sub_of_eps_is_eps} to each pair $(X'_{k+1}, X'_j)$, $1 \leq j \leq k$.
\end{proof}

Now we will prove \Cref{lem:gen_small_bad} using \Cref{l:tool2}.

\begin{proof}[Proof of \Cref{lem:gen_small_bad}]
Let $\hat{\beta}$ satisfy
\[
  \hat{\beta}^{(1-\eps')\eps'/2} 4^{2\ell} \leq \frac{\beta}{8}
  .
\]
We apply \Cref{l:tool2} with $\ell \leftarrow \ell-1$, $\beta \leftarrow \hat{\beta}$ and $\eps'\leftarrow \eps'$ to obtain constants $\eps_{\ell - 1} \leftarrow \eps$ and $C \leftarrow C$. We prove the lemma for
\[
  \eps = \min\left\{\eps_{\ell - 1}, \frac{\eps'}{4\ell}, \eps'^2\right\}
  \text{ and }
  C=C.
\]
We verify the conditions in~(\ref{eq: preconditions key lemma}) of \Cref{l:tool2}. It follows from $m \geq 2 n^{3/2} \sqrt{\log n}$ that
\[
  x = (1 - \eps')\frac{m}{n} 
  = \frac{m^2}{n^3 \log n}(1 - \eps')\frac{n^2}{m}\log n 
  \geq \frac{n^2}{m}\log n=\frac{\log n}{d}
\]
for $n$ sufficiently large.
As $d = m/n^2$, we also have 
\[
d \ge Cd_0 \enspace, \quad \eps' d \le  d' \le (1-\eps') d.
\]
Let $\G'(K_\ell-e, n, m, \eps) \seq \G(K_\ell-e,n,m,\eps)$ denote the subfamily of graphs $G$ that satisfy the following property: the number of $(\ell-2)$-tuples $(X_3,\ldots,X_\ell)$ such that for each $3 \leq i \leq \ell$, $X_i \seq V_i$, $|X_i| = (1+\eps')x$, and there exists $X'_i \seq X_i$ of size $x$ such that the induced graph $G[V_2,X'_3,\ldots,X'_\ell]$ contains a member of the family
\[
  \G(\ell - 1, n,x,\dots,x,d', \eps') \setminus \B(\ell - 1, n,x,\dots,x,d',\eps')
\]
as a subgraph, is at least $(1 - \hat{\beta}^x){n \choose (1 + \eps')x}^{\ell - 2}$. \Cref{l:tool2} applied to $(V_2,\ldots,V_\ell)$ with $x_2=n, x_i=x$ for $3 \le i \le \ell$ yields
\[
  |\G'(\ell, n, m, \eps)| 
  \geq (1 - \hat{\beta}^m) {n^2\choose m}^{{\ell - 1 \choose 2}}
  	{n^2\choose m}^{\ell - 2}
  \geq \left[1 - \left(\frac{\beta}{2}\right)^m\right] 
  	{n^2\choose m}^{{\ell \choose 2}-1}
  .
\]
Observe that the factor ${n^2\choose m}^{\ell - 2}$ accounts for the number of ways to distribute the edges between $V_1$ and $\cup_{i=3}^\ell V_i$. We constructively count all graphs that belong to $\G'(K_\ell-e, n, m, \eps)$, but violate the conditions of \Cref{lem:gen_small_bad}, and show that there are only a few of them. Firstly, we select all the edges between $V_i$ and $V_j$ for all $2 \leq i < j\leq \ell$. There are at most
\[
  {n^2\choose m}^{\ell - 1 \choose 2}
\]
possibilities. Secondly, for all the vertices $v \in V_1$, we select the degrees $d_j(v)$ into $V_j$ for $j\geq 3$. There are at most $(n+1)^{\ell n}\leq 2^m$ possibilities for sufficiently large $n$.
By \Cref{prop:reg_degree} and since we are constructing an $(\eps)$-regular graph between $V_1$ and $V_j$, for at least $(1 - \eps\ell)n$ vertices $v \in V_1$, we have to choose degrees $d_j(v) \geq (1 - \eps) m / n$ into all the sets $V_j$ for $j \ge 3$. Now we choose a set of at least $\eps' n$ vertices whose neighbourhood (together with $V_2$) does not contain a graph in $\G(\ell - 1, x_2,\dots,x_\ell, \eps') \setminus \B(\ell - 1, x_2,\dots,x_\ell)$ as a subgraph. There are at most $2^n \leq 2^m$ possibilities to choose these vertices. We denote by $A$ the set of all such vertices that additionally have a degree $\simeps{\eps}m / n$ into each set $V_j$ for $3 \leq j \leq \ell$. Note that $|A| \geq (\eps' - 2\ell\eps)n \geq (\eps'/2)n$ and that each vertex in $A$ has degree greater than
\[
  (1 - \eps) \frac{m}{n} \geq (1 - \eps'^2)\frac{m}{n} 
  = (1 + \eps')(1 - \eps')\frac{m}{n} = (1 + \eps')x.
\]
Now we select the neighbourhoods for the vertices in $V_1 \setminus A$. There are at most ${n \choose d_j(v)}$ possibilities for each vertex $v$ to choose its neighbourhood in $V_j$, where $d_j(v)$ is the already fixed size of the neighbourhood of $v$ in $V_j$. For all vertices in $A$, we first choose a set of size $(1 + \eps')x$ in each partition class $V_3,\ldots, V_\ell$. We require that these sets do not contain subsets of size $x$ that induce a graph which contains a member of $\G(\ell - 1, x_2,\dots,x_\ell, \eps') \setminus \B(\ell - 1, x_2,\dots,x_\ell)$. Since we consider graphs in $\G'(\ell, n, m, \eps)$, there are at most $\hat{\beta}^{x}{n \choose (1 + \eps')x}^{l-2}$ ways to choose such sets $X_i$, $3 \leq i \leq \ell$, in the neighbourhoods of each vertex in $A$. Now we have to choose the remaining neighbours for every vertex $v\in A$. There are at most $\prod_{j=3}^\ell{n - (1 + \eps')x \choose {d_j(v)- (1 + \eps')x}}$ ways to do this. The number of ways to select the neighbourhoods of vertices in $A$ is thus at most
\begin{equation*}
\begin{split}
  &\prod_{v\in A} \left(\hat{\beta}^{x} {n\choose (1 + \eps')x}^{\ell-2}
  \prod_{j=3}^\ell {n-(1 + \eps')x\choose {d_j(v)-(1 + \eps')x}}\right) \\
  &\leBy{eq:binomial coefficient 2}\; \hat{\beta}^{x\eps'n / 2} 
    \left(\prod_{v \in A}\prod_{j=3}^{\ell}4^{d_j(v)}{n\choose d_j(v)}\right)
  \leByM{d_j(v)\leq 2m/n} \hat{\beta}^{x \eps' n / 2} 4^{2\ell m}
    \left(\prod_{v \in A}\prod_{j=3}^{\ell}{n\choose d_j(v)}\right)
  .
\end{split}
\end{equation*}
We conclude from $x=(1 - \eps')m/n$ that there are at most
\[
  \left(\hat{\beta}^{(1 - \eps')\eps'/2} 4^{2\ell}\right)^m\left(\prod_{v\in V_1}
  \prod_{j=3}^{\ell}{n\choose d_j(v)}\right)
  \leq \left(\frac{\beta}{8}\right)^m{n^2\choose m}^{\ell-2}
\]
ways to select the neighbourhoods of the vertices in $V_1$. Taking the graphs in $\G(\ell, n, m, \eps) \setminus \G'(\ell, n, m, \eps)$ into account, we proved that there are at most 
\[
  \left[\left(\frac{\beta}{2}\right)^m + (2^m)^2 \left(\frac{\beta}{8}\right)^m\right]{n^2\choose m}^{\ell-2}{n^2\choose m}^{\ell-1\choose 2} \leq \beta^m{n^2\choose m}^{{\ell \choose 2}-1}
\]
graphs that violate the conditions of the lemma. 
\end{proof}




\section*{Acknowledgements}

The author would like to thank Stefanie Gerke for valuable suggestions and careful reading of the manuscript. The author was supported by a Royal Thai Government Doctoral Studentship.

\bibliographystyle{siam}
\bibliography{ref}

\end{document}